\documentclass[12pt,a4paper]{amsart}

\usepackage{comment}

\makeatletter
\usepackage[utf8]{inputenc}

\usepackage{amsmath}
\usepackage{amssymb}
\usepackage[usenames,dvipsnames]{color}
\usepackage[all,cmtip]{xy}
\usepackage{fullpage}
\numberwithin{equation}{section}
\numberwithin{figure}{section}
\usepackage[all,cmtip]{xy}
\usepackage{hyperref}
\usepackage{longtable}
\usepackage{faktor}
\usepackage{paralist}
\usepackage{enumerate}
\usepackage{graphicx}
\usepackage{faktor}
\usepackage{tikz-cd}
\usepackage{tensor}
\usepackage{fvextra}
\usetikzlibrary{calc,intersections}

\newtheorem{theorem}{Theorem}[section]
\newtheorem{corollary}[theorem]{Corollary}
\newtheorem{proposition}[theorem]{Proposition}
\newtheorem{lemma}[theorem]{Lemma}

\theoremstyle{definition}
\newtheorem{definition}[theorem]{Definition}
\newtheorem{example}[theorem]{Example}

\newtheorem{problem}[theorem]{Problem}
\newtheorem{remark}[theorem]{Remark}

\newcommand{\NN}{\mathbb{N}}
\newcommand{\ZZ}{\mathbb{Z}}
\newcommand{\QQ}{\mathbb{Q}}
\newcommand{\FF}{\mathbb{F}}
\newcommand{\RR}{\mathbb{R}}

\newcommand{\PP}{\mathbb{P}}
\renewcommand{\AA}{\mathbb{A}}
\newcommand{\GG}{\mathbb{G}}
\newcommand{\EE}{\mathbb{E}}

\newcommand{\RRb}{\overline{\RR}}

\newcommand{\GL}{\mathop{\rm GL}\nolimits}
\newcommand{\PGL}{\mathop{\rm PGL}\nolimits}

\newcommand{\SL}{\mathop{\rm SL}\nolimits}
\newcommand{\Gal}{\mathop{\rm Gal}\nolimits}

\newcommand{\p}{\mathfrak{p}}
\newcommand{\OO}{\mathcal{O}}
\newcommand{\Spec}{\mathop{\rm Spec}}

\newcommand{\A}{\mathcal A}
\newcommand{\B}{\mathcal{B}}
\newcommand{\X}{\mathcal X}

\newcommand{\E}{\mathcal E}

\newcommand{\M}{\mathcal M}
\newcommand{\I}{\mathcal I}

\newcommand{\nr}{{\scriptstyle \rm nr}}

\newcommand{\rig}{{\rm rig}}

\newcommand{\iso}{\stackrel{\sim}{\to}}
\newcommand{\liso}{\stackrel{\sim}{\longrightarrow}}

\newcommand{\gen}[1]{\langle\,#1\,\rangle}
\newcommand{\abs}[1]{\lvert#1\rvert}
\newcommand{\norm}[1]{\lVert#1\rVert}

\newcommand{\lexp}[2]{\tensor[^#1]{#2}{}}

\newcommand{\N}{\mathcal{N}}

\begin{document}

	\title{Models of hypersurfaces and Bruhat-Tits buildings}

    \author{Kletus Stern}
    \address{Institute for Algebra and Number Theory, Ulm University}
    \email{kletus.stern@uni-ulm.de}

    \author{Stefan Wewers}
    \address{Institute for Algebra and Number Theory, Ulm University}
    \email{stefan.wewers@uni-ulm.de}

    \subjclass[2020]{14G20 (Primary), 14L30, 20E42, 14Q25, 14G22 (Secondary)}

    \keywords{Hypersurface models, semistable reduction, discrete valuation,
    Bruhat--Tits buildings, non-Archimedean geometry, Geometric Invariant Theory}

    \thanks{We thank Robert Nowak for helpful and inspiring discussions about this project.}

    \begin{abstract}
    	We propose a new approach to constructing semistable hypersurface models over a discretely valued field $K$. For each GIT-stable hypersurface $X/K$ we define a uniformly continuous “stability function” $\phi_X$ on the \emph{metric} Bruhat--Tits building $\B_K$ of $\PGL_{n+1}(K)$. The minimum locus of this function governs semistable hypersurface models: vertices in this locus are precisely the semistable models, and every rational minimizer becomes such a vertex after a finite extension. In residue characteristic $0$ this yields a practical solution of semistable reduction via minimization of $\phi_X$. More generally, we prove that among all finite extensions $L/K$ the minima of $\phi_X$ on $\B_L$ admit a global minimum, giving a new proof of semistable reduction without using GIT. We implement the resulting strategy for plane curves over $p$-adic number fields, and in a follow-up article apply it to compute the semistable reduction (in the sense of Deligne and Mumford) of smooth plane quartics.
    \end{abstract}

\maketitle

\section{Introduction}

\subsection{Models of hypersurfaces with respect to a discrete valuation}

Let $K$ be a field which is complete with respect to a discrete valuation
\[
v_K:K\to\QQ\cup\{\infty\}.
\]
We denote by $\OO_K$ the valuation ring of $v_K$, fix a prime element $\pi$ and let $k=\OO_K/(\pi)$ denote the residue field. We assume that $k$ is perfect.  We let $\Gamma_K:=v_K(K^\times)\subset\QQ$ denote the valuation group (a discrete subgroup of $(\QQ,+)$). As long as the field $K$ is fixed, we may always assume that $\Gamma_K=\ZZ$.

Let $n,d\geq 1$ be positive integers and $V$ a $K$-vector space of dimension $n+1$. We let
\[
W := K[V]_d
\]
denote the $K$-vector space of homogeneous polynomials of degree $d$ on $V$. We fix a nonzero element $F\in W\backslash\{0\}$. It defines a hypersurface of degree $d$,
\[
X := V(F)\subset \PP(V)\cong \PP_K^n,
\]
in projective space of dimension $n$. We wish to study integral models of $X$ with respect to the valuation $v_K$, {\em as a hypersurface}.

Given a {\em coordinate system} $\E=(x_0,\ldots,x_n)$ on $\PP(V)$, i.e.\ a $K$-basis of the dual space $V^*$, we can write
\[
F = \sum_{i} a_i x_0^{i_1}\cdots x_n^{i_n},
\]
with $a_i\in K$ and where $i=(i_0,\ldots,i_n)$ runs over all tuples with $i_j\geq 0$ and $i_0+\ldots+i_n=d$. After multiplication with a suitable constant, we
may assume that the form $F$ defining $X$ is {\em primitive}, i.e.\ that
\[
\min_i v_K(a_i) = 0.
\]
Then $F\in\OO_K[x_0,\ldots,x_n]$, and the Zariski closure of $X$ inside projective $n$-space over $\OO_K$ is the hypersurface defined by $F$,
\[
\X = V(F)\subset\PP^n_{\OO_K}.
\]
We call $\X$ a {\em hypersurface model} of $X$ with respect to $v_K$. The special fiber of $\X$ is the hypersurface of degree $d$
\[
\X_s = V(\overline{F}) \subset\PP^n_k,
\]
where $\overline{F}\in k[x_0,\ldots,x_n]_d$ is the reduction of $F$. We note that the hypersurface model $\X$, and therefore also its special fiber $\X_s$, depends on the choice of the coordinate system $\E$. The problem is to find one such choice which gives the `best' model.

\subsection{Semistable hypersurface models}

We assume, from now on, that  $d\geq 3$ and that the hypersurface $X$ is {\em stable} in the sense of GIT (Geometric Invariant Theory, see \cite{MumfordGIT}). For instance, this condition holds if $X$ is smooth (\cite[Proposition 4.2]{MumfordGIT}).

\begin{definition} \label{def:stable_model}
	A hypersurface model $\X$ of $X$ is called {\em stable} (resp.\ {\em semistable}) if its special fiber $\X_s\subset\PP^n_k$ is stable (resp.\ semistable) in the sense of GIT. See \S \ref{subsec:stability_definitions} for a precise definition. (It is here that we use the assumption that the residue field $k$ is perfect.)
\end{definition}

\begin{remark}
	\begin{enumerate}[(i)]
		\item
		Our terminology is not consistent with the terminology introduced by Koll\'ar in \cite{Kollar97}. A semistable (resp. stable) model in the sense of Definition \ref{def:stable_model} is semistable (resp.\ stable) in Koll\'ar's sense, but the converse is in general false. The semistable models in Koll\'ar's sense are what we call {\em minimal models}, see Corollary \ref{cor:stability_function_intro} below.
		\item
		If $K$ has positive characteristic and has a discrete valuation, then it cannot be perfect. In this case there seems to be a problem with the condition that $X$ is stable in the GIT sense, because the usual definition of \emph{(semi)stable} is problematic over a non-perfect field. A pragmatic way out is to use the numerical criterion of Hilbert-Mumford-Kempf (Theorem \ref{thm:numerical_criterion}) as a definition for (semi)stability. In fact, throughout this article we will use (semi)stability of hypersurfaces only via this criterion.
	\end{enumerate}
\end{remark}

In general, no semistable model may exist, but we have the following `semistable reduction theorem'.

\begin{theorem} \label{thm:semistable_reduction}
	There exists a finite extension $L/K$ such that the base change of $X$, $X_L\subset\PP^n_L$, has a semistable hypersurface model with respect to the unique extension of $v_K$ to $L$.
\end{theorem}

\begin{proof}
	This is a general feature of GIT. See e.g.\ \cite[Lemma 5.3]{MumfordSPV}, \cite[Proposition 2]{Burnol}, or \cite[Theorem 3.3]{LLLR}. We will give a rather different proof of this theorem, via Theorem \ref{thm:global_minimum} below.
\end{proof}

Here is the problem motivating this work:
\begin{problem} \label{prob:semistable_model}
	Given a stable hypersurface $X\subset \PP^n_K$ over a field $K$ which is complete with respect to a discrete valuation $v_K$, construct a finite extension $L/K$ and a semistable hypersurface model $\X$ of $X_L$.
\end{problem}

To our knowledge, no practical algorithm for solving this problem is known, except for $n=1$ (binary forms) and $n=2$, $d=3$ (plane cubics). Using the results of this article, the first author has developed and implemented an algorithm which solves Problem \ref{prob:semistable_model} for arbitrary plane curves ($n=2$) of degree $d\geq 3$ over $p$-adic number fields.  

\vspace{2ex}
Consider the case $n=2$. Then $X\subset\PP^2_K$ is a plane curve and there is a connection to the alternative notion of semistable models of curves in the sense of Deligne and Mumford (\cite{DeligneMumford69}); we shall call such models {\em geometrically semistable}. For simplicity, assume that $X$ is smooth.

For plane cubics ($d=3$) the two definitions of semistability more or less agree: a hypersurface model $\X\subset\PP^2_{\OO_K}$ is semistable in the sense of Definition \ref{def:stable_model} if and only if its special fiber $\X_s\subset\PP^2_k$ is reduced, irreducible and has at most one ordinary double point as singularity. It is stable if and only if $\X_s$ is smooth. See \cite[\S 4.2]{MumfordGIT}. Therefore, a solution to Problem \ref{prob:semistable_model} is given by the well known procedure to compute the (geometric) semistable reduction of elliptic curves.

For $d\geq 4$, the genus of $X$ is $g=(d-1)(d-2)/2\geq 3$ and the two definitions of semistability start to diverge. A hypersurface model of $X$ which is semistable in the sense of Definition \ref{def:stable_model} is in general not geometrically semistable. By the Semistable Reduction Theorem of Deligne and Mumford, there exists an extension $L/K$ and a model $\X$ of $X_L$ which is geometrically semistable; but in general, $\X$ is not a hypersurface model, so it doesn't make sense to apply Definition \ref{def:stable_model}. The precise relation between GIT-semistable plane model and geometric semistable models of plane curves has been studied a lot, see e.g.\ \cite{hassett1999stable}, \cite{LLLR}, \cite{van2025reduction} and, most recently, \cite{MaxMaster}.

The main motivation for us to study GIT-semistability of plane curves is the hope that it will help to compute the geometric semistable reduction of curves of genus $g\geq 3$. In fact, for non-hyperelliptic curves over local fields of small residue characteristic, this is an open problem in computational arithmetic geometry (\cite{oberwolfach}). Building on the results of this article, we solve this problem for smooth plane quartics with non-hyperelliptic reduction, see \cite{SSW}.

\vspace{2ex}
The goal of this article is to lay the foundation for a new and general approach towards Problem \ref{prob:semistable_model}. A more detailed analysis of the case of plane curves, and an algorithm solving Problem \ref{prob:semistable_model} in this case will be given in the first author's thesis (\cite{KletusDiss}). An implementation of this algorithm is available at \cite{KletusGitHub}. We will give a brief demonstration in \S \ref{subsec:example1}.

\subsection{The stability function} \label{subsec:stability_function_intro}

The first observation motivating our approach is that the set of isomorphism classes of hypersurface models of $X$ can naturally be identified with the set of vertices of the Bruhat-Tits building $\B_K$ of the group $\PGL_{n+1}(K)$, defined below.

We let $\B^0_K$ denote the set of homothety classes of full $\OO_K$-lattices $L\subset V^*$ 
(two lattices $L_1,L_2\subset V$ are homothetic if there exists $c\in K^\times$ such that $L_2=c\cdot L_1$). Then $\B_K$ is defined as the simplicial complex with underlying vertex set $\B_K^0$, whose $s$-simplices are the sets of the form $\{[L_0],\ldots,[L_s]\}$ such that
\[
L_0\subsetneq L_1\subsetneq \ldots\subsetneq L_s\subsetneq \pi^{-1}L_0.
\]
It is well known that $\B_K$ is a {\em euclidean building}\footnote{Using the traditional definition of {\em building}, this is true only if $K$ is a local field, i.e.\ if the residue field is finite.}, see e.g.\ \cite[\S 3]{ASENS_2011} or \cite{parreau_immeubles}. In particular, $\B_K$ has an underlying geometric realization which, as a metric space, is the union of euclidean affine spaces of dimension $n$ called {\em apartments}. In the following, we use the letter $\B_K$ to denote this metric space.

On each apartment $\A\subset\B_K$, the set of vertices $\A^0=\A\cap\B_K^0$ forms a
lattice in the underlying affine space. Besides the vertices, $\B_K$ contains a
distinguished dense subset $\B_K(\QQ)$ of \emph{rational points}, which may be
characterized by having rational coordinates with respect to some (equivalently,
any) apartment.

We have defined a hypersurface model $\X$ of $X$ by the choice of a $K$-basis $\E=(x_0,\ldots,x_n)$ of
\[
V^* = H^0(\PP^n_K,\OO(1)).
\]
It is easy to see that the isomorphism class of $\X$ only depends on the homothety class $[L]\in\B_K^0$ of the lattice
\[
L:=\gen{x_0,\ldots,x_n}_{\OO_K}.
\]
This defines a natural bijection between the set of isomorphism classes of hypersurface models of $X$ with the set $\B_K^0$.

Koll\'ar (\cite{Kollar97}) and  Elsenhans and Stoll (\cite{ES}) study the problem, related to Problem \ref{prob:semistable_model}, of finding {\em minimal hypersurface models} of $X$. From our point of view, this amounts to computing the minimum of a certain function
\[
\phi_X^\circ:\B_K^0\to\ZZ,
\]
which tracks the valuation of invariants of the hypersurface $X$ with respect to the coordinate system defining a model. Our second observation is that this function extends naturally to a continuous function on the full metric space $\B_K$ with very nice properties. 
Using this function, the minimization problem governing semistable reduction is naturally a
\emph{continuous} optimization problem on the metric space $\B_K$, rather than a discrete
problem on its vertex set.

More precisely, we have the following result.

\begin{proposition} \label{prop:stability_function_intro}
	Let $X\subset\PP^n_K$ be a stable hypersurface of degree $d$. Then there exists a uniformly continuous  function
	\[
	\phi_X:\B_K\to\RR,
	\]
	with the following properties.
	\begin{enumerate}[(i)]
		\item
		The restriction of $\phi_X$ to any apartment $\A\subset\B_K$ (which is an euclidean affine space of dimension $n$) is piecewise affine, convex and radially unbounded. By the latter we  mean that for every sequence of points $x_0,x_1,\ldots\in\A$ we have
		\[
		d(x_0,x_i)\to\infty \quad\Rightarrow\quad \phi_X(x_i)\to\infty.
		\]
		\item
		The function $\phi_X$ achieves a global minimum $m_X$ in a rational point of $\B_K$.
		\item
		Set
		\[
		M_X:= \{ b\in\B_K \mid \phi_X(b)=m_X\}.
		\]
		Let $b\in\B_K^0$ be a vertex, corresponding to the hypersurface model $\X$. Then $\X$ is semistable if and only if $b\in M_X$. Moreover, $\X$ is stable if and only if $M_X=\{b\}$ over any finite extension $L/K$.
		
		In particular, $X$ has a semistable model over $\OO_K$ if and only if the set $M_X$ contains a vertex of $\B_K$.
	\end{enumerate} 
\end{proposition}

We call the function $\phi_X$ the {\em stability function} associated to the hypersurface $X$.

\vspace{2ex}
Assuming a solution to another algorithmic problem (Problem~\ref{prob:find_instability}), the explicit definition of the stability function $\phi_X$ leads to an algorithm for finding a rational point $b\in\B_K$ where $\phi_X$ attains its minimum $m_X$
(see \S~\ref{subsec:finding_minimum}). In \S \ref{sec:stability} we solve this problem for $n=2$. This yields an effective procedure to detect
semistable models of plane curves when they exist.
If no semistable model exists over $K$, the algorithm above does not directly determine
a minimal model. While this can be remedied by a minor modification, it offers little
advantage over existing methods such as those of Elsenhans and Stoll~\cite{ES}.

The definition of $\phi_X$ and the proof of Proposition \ref{prop:stability_function_intro} is given in \S \ref{sec:stability_function}, and is rather straightforward. The definition of the function $\phi_X$ on the set of vertices of $\B_K$ is already suggested by Koll\'ar, and its extension to all of $\B_K$ is suggested by the identification of $\B_K$, as a metric space, with a space of certain valuations, well known by work of Goldman and Iwahori (\cite{GoldmanIwahori}). Part (i) of the proposition follows immediately from making this definition explicit. Part (ii) follows from (i) and the fact, essentially proved by Koll\'ar, that $\phi_X$ is bounded from below. For this argument, and for this argument only, we use Geometric Invariant Theory as a tool.\footnote{And this is not really necessary: in \S \ref{sec:stable_minimum} we prove a much stronger result, which implies that $\phi_X$ is bounded from below, without using Geometric Invariant Theory.} Part (iii) is again clear from Koll\'ar's work, as is the following corollary.

\begin{corollary}\label{cor:stability_function_intro}
	A semistable model is also minimal.
\end{corollary}

The converse of Corollary~\ref{cor:stability_function_intro} is false: by work of Koll\'ar~\cite{Kollar97}, a minimal hypersurface model always exists, but a semistable model need not exist over $K$. As Proposition \ref{prop:stability_function_intro} (iii) shows, the obstruction
to the existence of a semistable model over $K$ is \emph{not} the lack of minimality,
but rather the fact that the minimum locus $M_X$ of the stability function may fail
to contain a vertex. This observation motivates the discussion in the next subsection,
where we study how passing to finite extensions of $K$ affects the set $M_X$.

\medskip
We record two additional remarks for later reference.
\begin{itemize}
	\item
	If the residue field of $K$ is finite, then $\B_K$ is locally compact and $M_X$ contains only finitely
	many vertices; hence $X$ admits only finitely many non-isomorphic semistable models
	(cf.~\cite[Theorems~4.1.2 and~5.2.3]{Kollar97}).
	\item
	If the residue field is infinite, the set $M_X$ need not be compact and may
	contain infinitely many vertices; an explicit family of such examples is given in
	\S~\ref{subsec:example2}.
\end{itemize}

\subsection{Finding the right field extension} \label{subsec:finding_right_extension}

The converse of Corollary~\ref{cor:stability_function_intro} fails because the minimum
locus $M_X$ of the stability function $\phi_X$ need not contain a vertex. However,
the proof of Proposition~\ref{prop:stability_function_intro} shows that $M_X$ always
contains at least one \emph{rational point}. A key feature of the Bruhat--Tits building
is that every rational point of $\B_K$ becomes a vertex after passing to a sufficiently
ramified finite extension $L/K$.

More precisely, for any finite extension $L/K$ there is a natural isometric embedding
\[
\iota:\B_K \hookrightarrow \B_L,
\]
and for $L/K$ sufficiently ramified the image of any rational point of $\B_K$ is a
vertex of $\B_L$.

From the definition of the stability function $\phi_X$ it is obvious that it is compatible with such base change. It follows that the obstruction to the existence of a semistable model
over $K$ is purely \emph{integrality}: the minimum locus $M_X$ may fail to contain a vertex,
even though it always contains a rational point (Proposition~\ref{prop:stability_function_intro}(ii)).

\vspace{2ex}
In the case of residue characteristic zero, this integrality issue is the only one.

\begin{theorem} \label{thm:tame_case_intro}
	Assume that the residue field $k$ has characteristic $0$.
	Let $b\in \B_K(\QQ)$ be a rational point where $\phi_X$ attains its minimum $m_X$.
	Let $L/K$ be a finite extension such that the image $\iota(b)$ is a vertex of $\B_L$. Then $\iota(b)$ corresponds to a semistable hypersurface model of $X_L$ over $\OO_L$.
\end{theorem}

Equivalently, in residue characteristic zero the problem of constructing a semistable hypersurface model reduces to the explicit minimization of the stability function on the \emph{original} building $\B_K$, together with the purely formal step of passing to a finite extension $L/K$ that turns a rational minimizer into a vertex.

The proof of Theorem \ref{thm:tame_case_intro} combines Proposition~\ref{prop:stability_function_intro} with the fixed-point theorem for tamely ramified extensions of Rousseau--Prasad (\cite{prasad01}).

\vspace{2ex}
If the residue field $k$ has positive characteristic (the case we are mainly interested in) the problem of finding the correct field extension becomes much more difficult. To better understand the nature of the problem, the following terminology is useful. Let us denote by $m_{X,L}$ the minimal value of $\phi_X$ on $\B_L$.

\begin{definition}\label{def:stable_minimum_intro}
	A point $b\in \B_K$ is called a \emph{stable minimizer} 
	if for every finite extension $L/K$ the image $\iota(b)\in \B_L$ under the natural embedding
	\[
	\iota:\B_K\hookrightarrow\B_L
	\]
	is still a minimizer of $\phi_X$ on $\B_L$, i.e.
	\[
	\phi_X(\iota(b)) = m_{X,L}.
	\]
	The minimal value $m_{X,L}$ is then called the {\em stable minimum} of $\phi_X$.
\end{definition}

Here is what we can directly deduce from Proposition \ref{prop:stability_function_intro}.

\begin{proposition}\label{prop:stable_minimum}
	\begin{enumerate}[(i)]
		\item 
		Let $L/K$ be a finite extension. If $X$ has a semistable hypersurface model over $L$ then $m_{X,L}$ is the stable minimum of $\phi_X$.	
		\item
		Let $b\in\B_K(\QQ)$ be a rational stable minimizer, and let $L/K$ be a finite extension such that $\iota(b)\in\B_L^\circ$ is a vertex. Then the hypersurface model of $X$ corresponding to $\iota(b)$ is semistable.
	\end{enumerate}
\end{proposition}

See \S \ref{subsec:tame_case} for the (easy) proof. 
The proposition shows that the existence of a semistable hypersurface model after a suitable field extension (Theorem \ref{thm:semistable_reduction}) is equivalent to the following result.

\begin{theorem} \label{thm:global_minimum}
	The set $\{m_{X,L}\}$, where $L/K$ runs over all finite field extensions, has a minimum.
\end{theorem}

We prove this theorem in \S \ref{sec:stable_minimum}.
What is interesting is that our proof does not use Geometric Invariant Theory. Instead, it combines relatively basic techniques from non-archimedean analysis and polyhedral geometry. Using this novel approach, the first author has devised and implemented a practical algorithm which solves Problem \ref{prob:semistable_model} for plane curves, even when the residue field $k$ has positive characteristic (\cite{KletusDiss}, \cite{KletusGitHub}).

\section{Stability of hypersurfaces} \label{sec:stability}

\subsection{Definitions and conventions}
\label{subsec:stability_definitions}

In this section, $K$ is a perfect field, and $n,d\geq 1$ are positive integers. We fix a $K$-vector space $V$ of dimension $n+1$, which gives rise to the algebraic groups
\[
G:=\SL(V)\subset \tilde{G}:=\GL(V).
\]
We also let
\[
W:= K[V]_d
\]
denote the space of homogenous polynomials of degree $d$ on $V$. The natural action of $\tilde{G}$ on $V$ induces an action on $W$. We use the following convention: for $F\in W$ and $g\in\tilde{G}$ we set
\[
\lexp{g}{F} := F\circ g^{-1}.
\]
So, as representations of the group $\tilde{G}$, the vector space $W$ is the $d$th symmetric power of the contragredient of $V$.

A nonzero element $F\in W\backslash\{0\}$ defines a hypersurface $X=V_+(F)\subset\PP(V)$. It corresponds exactly to the $K$-rational point $[F]\in\PP(W)$. We write $O(F)\subset \AA(W)$ for the $G$-orbit of $F$, i.e.\ the image of the morphism of $K$-varieties
\[
G\to \AA(W), \quad g\mapsto \lexp{g}{F},
\]
and $G_F\subset G$ for the stabilizer of $F$. It is known that $O(F)\subset \AA(W)$ is locally closed in the Zariski topology, and that $G_F\subset G$ is a closed subgroup scheme (see \cite[\S 9.c]{Milne_iAG}).

\begin{definition} \label{def:stability}
	The hypersurface $X=V_+(F)$ is called {\em semistable} if the Zariski closure of the $G$-orbit of $F$ does not contain the zero vector,
	\[
	0\not\in\overline{O(F)}\subset \AA(W).
	\]
	It is called {\em stable} if the $G$-orbit $O(F)$ is closed, and the stabiliser $G_F$ is finite.
	
	If $X$ is semistable, but not stable, we call it {\em not properly stable}. If it is not semistable, we call it {\em unstable}.
\end{definition}

These definitions only depend on $X$, and are stable under base change of $X$ by a field extension $L/K$. It follows that the property of being (semi)stable are geometric, i.e.\ depend only on the base change of $X$ to an algebraic closure of $K$.

\subsection{The numerical criterion} \label{subsec:numerical_criterion}

As before, we fix a hypersurface $X=V_+(F)$, with $F\in W\backslash\{0\}$. Given a one-parameter subgroup
\[
\lambda:\GG_m=\Spec K[t,t^{-1}]\hookrightarrow G,
\]
the action of $\GG_m$ on $W$ induced by $\lambda$ corresponds to a decomposition
\[
W = \oplus_{m\in\ZZ} W_m^\lambda,
\]
where $W_m^\lambda\subset W$ is the eigenspace of the character $\chi_m:\GG_m\to\GG_m$ defined by $\chi_m(t)=t^m$.
Let $F_m\in W_m^\lambda$ denote the $\chi_m$-component of $F\in W$. We set
\[
\mu(F,\lambda) := \max\{ m\in \ZZ \mid F_m\neq 0\}.
\]
\begin{theorem}[Hilbert-Mumford-Kempf] \label{thm:numerical_criterion}
	The hypersurface $X=V_+(F)$ is semistable if and only if
	\[
	\mu(F,\lambda) \geq 0
	\]
	for all one-parameter subgroups $\lambda$. Moreover, $X$ is stable if and only if $\mu(F, \lambda) > 0$ for all one-parameter subgroups $\lambda$ defined over $\overline{K}$.
\end{theorem}

\begin{proof}
	If $K$ is algebraically closed, this is the famous result by Mumford, see \cite[Chapter 2, \S 1]{MumfordGIT}. As Kempf has shown (see \cite[Theorem 4.2]{KempfInstability}), it suffices to look at the one-parameter subgroups defined over $K$ (it is here where we use the assumption that $K$ is perfect!).
\end{proof}

We wish to make the criterion from Theorem \ref{thm:numerical_criterion} more explicit. Let $\lambda:\GG_m\to G$ be a one-parameter subgroup. Then $\lambda$ is contained in a  maximal torus of $G$. This means that there exists a basis $\E=(x_0,\ldots,x_n)$ of $V^*$ and integers $w_0,\ldots,w_n\in\ZZ$ such that $w_0+\ldots+w_n=0$ and
\begin{equation} \label{eq:ops_lambda}
	\lambda(t) = \begin{pmatrix}
		t^{w_0} &        &          \\
		& \ddots &          \\
		&        & t^{w_n}
	\end{pmatrix},
\end{equation}
where $g=\lambda(t)\in G$ is represented as a matrix via the contragredient representation, i.e.\ we have
\[
\lexp{g}{x}_i = t^{w_i}x_i, \quad i=0,\ldots,n.
\]
Write
\begin{equation}
	F = \sum_{i} a_i x_0^{i_0}\cdots x_n^{i_n},
\end{equation}
with $a_i\in K$, and where $i$ runs over all tuples $i=(i_0,\ldots,i_n)$ with $i_j\geq 0$ and $i_0+\ldots+i_n=d$.
Set
\[
I_{F,\E} := \{ i \mid a_i\neq 0\}.
\]
Then
\begin{equation} \label{eq:stability2}
	\mu(F,\lambda) = - \min_{i\in I_{F,\E}} \gen{i,w}, \quad \text{where}\;
	\gen{i,w} := i_0w_0+\ldots+i_nw_n.
\end{equation}

In this way, we obtain the following explicit criterion for a hypersurface to be unstable. We use the following terminology: an element $w=(w_0, \ldots, w_n) \in \ZZ^{n+1}$ is called a {\em weight vector}. We call $w$ {\em balanced} if $w_0+\ldots+w_n=0$, and {\em ordered} if $w_0\geq \ldots\geq w_n$.

\begin{corollary} \label{cor:instability}
	The hypersurface $X=V_+(F)$ of degree $d$ is unstable if and only if there exists a basis $\E$ of $V^*$ and a balanced and ordered weight vector $w=(w_0,\ldots,w_n)$ such that
	\[
	\forall\, i\in I_{F,\E}:\;\gen{i,w} > 0.
	\]
\end{corollary}

\begin{definition} \label{def:instability}
	A one-parameter subgroup $\lambda$ of $G$ is called an {\em instability} for the hypersurface $X=V_+(F)$ if $\mu(F,\lambda) < 0$. We also call the corresponding pair $(\E,w)$ from Corollary \ref{cor:instability} an instability for $X$.
\end{definition}

\subsection{Finding instabilities}

In Problem \ref{prob:semistable_model} we start with a hypersurface $X=V_+(F)$ over a discretely valued field $K$ which we assume to be stable. In our approach for finding a semistable model of $X$ an essential step is to (a) check whether the special fiber $\X_s$ of a given hypersurface model $\X$ is semistable and (b) if it is not, find an explicit instability for $\X_s$. Here $\X_s$ is a hypersurface defined over the residue field of $K$. Therefore, any solution to Problem \ref{prob:semistable_model} using our method necessarily depends on a solution to the following problem.

\begin{problem} \label{prob:find_instability}
	Given a hypersurface $X=V_+(F)\subset\PP(V)$ of degree $d$, find an explicit instability $\lambda$ for $X$, or decide that none exists (i.e.\ that $X$ is semistable).
\end{problem}

There seems to exist an extensive literature on the related problem of finding sufficient criteria for $X$ to be semistable (see e.g.\ the recent article \cite{Mordant2023} and the references therein), but we are not aware of similar efforts for finding necessary criteria, and explicit instabilities, as demanded by Problem \ref{prob:find_instability}.

Nevertheless the case of plane curves (i.e.\ the case $n=2$) is relatively easy and seems to be well known to experts. For lack of an adaquate reference, we state the main result in a form suitable for our purposes.

%
%
%
%


\begin{proposition} \label{prop:plane_curves}
	Let $n=2$ and $X=V_+(F)\subset\PP(V)$ be a plane curve of degree $d\geq 3$. Assume that $X$ is unstable.
	\begin{enumerate}[(i)]
		\item
		At least one of the following statements holds.
		\begin{enumerate}[(a)]
			\item
			The curve $X$ contains a line $L\subset X$ with multiplicity $m>d/3$.
			\item
			The curve $X$ contains a $K$-rational point $P\in X$ with multiplicity $m>2d/3$.
			\item
			The curve $X$ contains a line $L\subset X$ with multiplicity $m>0$. Furthermore, if we write $X=m\cdot L+Y$, with $L\not\subset Y$, then $L$ and $Y$ intersect in a $K$-rational $P$ with intersection multiplicity $>(d-m)/2$.
			\item
			The curve $X$ contains a $K$-rational point $P\in X$ with multiplicity $m>d/2$, and the tangent cone of $X$ in $P$ contains a line $L$ with multiplicity $>m/2$.
		\end{enumerate}
		\item
		Assume that Statement (a) holds for a line $L$. Let $\E=(x_0,x_1,x_2)$ be a coordinate system such that $L=V_+(x_0)$. Then $(\E,w)$, with $w:=(2,-1,-1)$, is an instability for $X$.
		\item
		Assume that Statement (b) holds for a point $P$. Let $\E=(x_0,x_1,x_2)$ be a coordinate system such that $\{P\}=V_+(x_0,x_1)$. Then $(\E,w)$, with $w:=(1,1,-2)$, is an instability for $X$.
		\item
		Assume that Statement (a) and (b) are false. Then there exists a unique pair $(L,P)$ with the following properties. Firstly, one of the Statements (c) or (d) hold for $(L,P)$. Secondly, if $\E=(x_0,x_1,x_2)$ is a coordinate system such that $L=V_+(x_0)$ and $\{P\}=V_+(x_0,x_1)$, then there exists a balanced and ordered weight vector $w$ such that $(\E,w)$ is an instability.
	\end{enumerate}
\end{proposition}

\begin{proof}
	(compare with \cite[Corollary 7.4]{ES})
	Fix a coordinate system $\E=(x_0,x_1,x_2)$ and write
	\begin{equation} \label{eq:plane_curve1}
		F = \sum_{i\in I_{F,\E}} a_i x_0^{i_0}x_1^{i_1}x_2^{i_2}, \quad \text{with $a_i\neq 0$ for all $i\in I_{F,\E}$.}
	\end{equation}
	Let $w=(w_0,w_1,w_2)\in\ZZ^3$ be a balanced and ordered weight vector. Recall that this means that $w_0\geq w_1\geq w_2$ and $w_0+w_1+w_2=0$. Then $(\E,w)$ is an instability if and only if
	\begin{equation} \label{eq:plane_curve2}
		\gen{i,w} = i_0w_0+i_1w_1+i_2w_2 > 0, \quad \forall i\in I_{F,\E}.
	\end{equation}
	
	To prove (i), we assume that $(\E,w)$ is an instability, i.e.\ that \eqref{eq:plane_curve2} holds. We have to show that one of the Statements (a), (b), (c) or (d) holds. For the following case distinction, Figure \ref{fig:stability_picture} may be helpful.
	
	Assume first that $w_1=w_2$. Then we may assume $w=(2,-1,-1)$. For $i=(i_0,i_1,i_2)\in I_{F,\E}$ the inequality \eqref{eq:plane_curve2} is equivalent to $2i_0-i_1-i_2>0$. Using $i_0+i_1+i_2=d$ we find that
	\[
	i_0 > \frac{d}{3}, \quad \forall\,i\in I_{F,\E}.
	\]
	By \eqref{eq:plane_curve1} this means that $x_0^m\mid F$, for some $m>d/3$. In other words, the line $L:=V_+(x_0)$ is a component of $X$ of multiplicity $m>d/3$. This is Statement (a).
	
	Next we assume $w_0=w_1$. Then we may assume $w=(1,1,-2)$. For $i=(i_0,i_1,i_2)\in I_{F,\E}$ the inequality \eqref{eq:plane_curve2} is equivalent to
	\[
	i_0+i_1 > \frac{2d}{3}.
	\]
	This means that $F\in (x_0,x_1)^m$, for some $m>2d/3$; in other words, $P:=[0:0:1]$ is a point of $X$ of multiplicity $m>2d/3$. This is Statement (b).
	
	For the rest of the proof of (i) we may assume $w_0>w_1>w_2$. Assume in addition that $w_1\leq 0$. Write $F=x_0^mG$ such that $x_0\nmid G$. Then $m$ is the multiplicity of the line $L:=V_+(x_0)$ in $X$. If $m>d/3$ then we are in Case (a). So $m\leq d/3$. Let $k$ be minimal with the property that the monomial $x_1^kx_2^{d-m-k}$ occurs in $G$. Then $0\leq k\leq d-m$, $i:=(m,k,d-m-k)\in I_{F,\E}$, and $k$ is the intersection multiplicity of $L$ with $G:=V_+(G)$ in the point $P:=[0:0:1]$. From \eqref{eq:plane_curve2} we obtain the inequality
	\begin{equation} \label{eq:plane_curve3}
		mw_0 + kw_1+(d-m-k)w_2 > 0.
	\end{equation}
	If we assume that $m=0$, then \eqref{eq:plane_curve3}, together with the assumptions $w_0>w_1>w_2$, $w_1\leq 0$ and $0\leq k\leq d$, leads to a contradiction. It follows that $m>0$. We claim that $k>(d-m)/2$.
	
	To prove the claim, we write $k=(d-m)/2+l$. Then \eqref{eq:plane_curve3}, together with $w_0+w_1+w_2=0$, yields
	\begin{equation} \label{eq:plane_curve4}
		\begin{split}
			0 &< mw_0 + \frac{d-m}{2}w_1 + \frac{d-m}{2}w_2 + (w_1-w_2)l \\
			& = \left(m-\frac{d-m}{2}\right)w_0 + (w_1-w_2)l \\
			& = -\frac{d-3m}{2}w_0 + (w_1-w_2)l.
		\end{split}
	\end{equation}
	Using the assumptions $w_1>w_2$ and $m\leq d/3$, we conclude from \eqref{eq:plane_curve4} that
	\[
	l > \frac{(d-3m)w_0}{2(w_1-w_2)} \geq 0.
	\]
	This proves the claim $k>(d-m)/2$ and shows that Statement (c) holds.
	
	Finally, we assume $w_1>0$. Then $w_2=-(w_0+w_1)<0$. Therefore the monomial $x_2^d$ does no occur in $F$ (i.e.\ $(0,0,1)\not\in I_{F,\E}$), and so the point $P:=[0:0:1]$ lies on $X$. Let $m>0$ be the multiplicity of $P$ as a point on $X$. If $m>2d/3$ we are in Case (b), so we may assume $m\leq 2d/3$.
	
	Let $k$ be minimal such that the monomial $x_0^kx_1^{m-k}x_2^{d-m}$ occurs in $F$, where $0\leq k\leq m$. Then $k$ is the multiplicity of the line $L:=V_+(x_0)$ in the tangent cone of $X$ in $P$. Since  $(k,m-k,d-m)\in I_{F,\E}$, \eqref{eq:plane_curve2} gives the inequality
	\begin{equation} \label{eq:plane_curve5}
		k w_0 +(m-k)w_1 + (d-m)w_2 >0.
	\end{equation}
	We have to show that $m>d/2$ and $k>m/2$, i.e.\ that Statement (d) holds.
	
	Using $-w_2=w_0+w_1$ we can rewrite \eqref{eq:plane_curve5} as
	\begin{equation} \label{eq:plane_curve6}
		k(w_0-w_1) > (d-m)w_0 + (d-2m)w_1.
	\end{equation}
	Using $w_0>w_1$ and $k\leq m$ we obtain from \eqref{eq:plane_curve6} the inequality
	\begin{equation} \label{eq:plane_curve7}
		m(w_0-w_1) \geq k(w_0-w_1) >(d-m)w_0 + (d-2m)w_1.
	\end{equation}
	Reordering \eqref{eq:plane_curve7} gives
	\begin{equation} \label{eq:plane_curve8}
		(2m-d)w_0 > (d-m)w_1.
	\end{equation}
	Since $w_0,w_1>0$ and $m\leq d$, \eqref{eq:plane_curve8} implies $m>d/2$, as claimed. On the other hand, using the assumption $m\leq 2d/3$ and $w_0,w_1>0$ we can deduce from \eqref{eq:plane_curve6} the inequality
	\begin{equation} \label{eq:plane_curve9}
		k(w_0-w_1) > (d-\frac{2d}{3})w_0 + (d-\frac{4d}{3})w_1
		= \frac{d}{3}(w_0-w_1).
	\end{equation}
	Since $w_0>w_1$, \eqref{eq:plane_curve9} implies
	\[
	k > \frac{d}{3} \geq \frac{m}{2}.
	\]
	This completes the proof of Part (i) of the Proposition.
	
	For the proof of (ii) we assume that $L=V_+(x_0)$ and that $x_0^m\mid F$, for some $m>d/3$. Then the same calculation as in the proof of (i) shows that $(\E,w)$, with $w:=(2,-1,-1)$, is an instability. The proof of (iii) is similar. We assume $P=[0:0:1]$ and that $F\in (x_0,x_1)^m$, for some $m>2d/3$. Then one can check directly that $(\E,w)$, with $w:=(1,1,-2)$, is an instability. See also Figure \ref{fig:stability_picture}.
	
	The proof of (iv) is more delicate, and we postpone it to the end of \S \ref{subsec:spherical_complex}. One problem is that Statements (c) and (d) do not uniquely determine the pair $(L,P)$. And if we have found the correct pair $(L,P)$ which is `responsible' for the instability, it is not obvious that any coordinate system $\E$ as in (iv) works.
\end{proof}

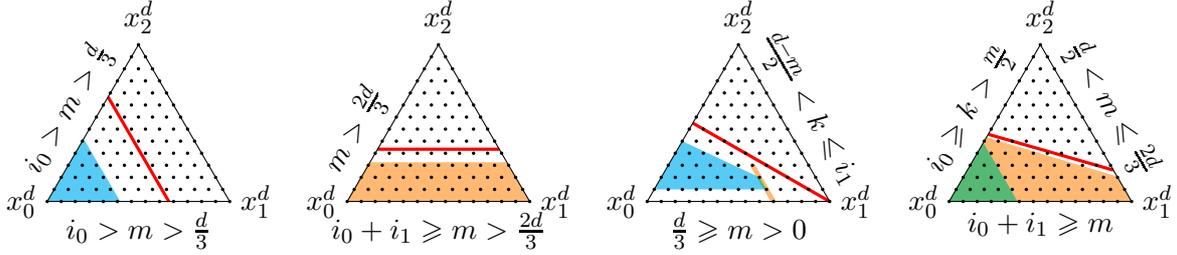
\begin{figure}
	\[
	\begin{tikzpicture}[scale=0.8, font = \small]
		

		
		\pgfmathsetmacro{\L}{3}	
		\pgfmathsetmacro{\d}{13}	
		\pgfmathsetmacro{\hstep}{\L/\d}
		\pgfmathsetmacro{\vstep}{\L*sqrt(3)/(2*\d)} 	
		

		
		\begin{scope}[blend mode=darken]
			
			\fill[cyan!55!white] (0,0) -- (\L*2/5,0) -- (\L/2*2/5,{\L/2*sqrt(3)*2/5});
			
		\end{scope}
		
		\draw[red,very thick] (\L*2/3,0) -- (\L/2*2/3,{\L/2*sqrt(3)*2/3});
		

		
		\foreach \i in {0,...,\d}
		\pgfmathsetmacro{\k}{\d-\i}	
		\foreach \j in {0,...,\k}
		\fill (\hstep/2*\i + \hstep*\j,\vstep*\i) circle (0.8pt);
		

		
		\draw (0,0) -- (\L,0) -- (\L/2,{\L/2*sqrt(3)}) -- (0,0);
		

		\fill (0,0) node[left]{$x_0^d$};
		
		\fill (\L,0) node[right]{$x_1^d$};
		
		\fill (\L/2,{\L/2*sqrt(3)}) node[above]{$x_2^d$};

		\fill (\L*1/2,0) node[below]{$i_0 > m > \frac{ d }{ 3 }$};
		
		\fill (\L/2*1/2,{\L/2*sqrt(3)*1/2}) node[rotate=60,above]{$i_0 > m > \frac{ d }{ 3 }$};

	\end{tikzpicture}
	\hspace{0.2cm}
	\begin{tikzpicture}[scale=0.8, font = \small]
		

		
		\pgfmathsetmacro{\L}{3}	
		\pgfmathsetmacro{\d}{13}	
		\pgfmathsetmacro{\hstep}{\L/\d}
		\pgfmathsetmacro{\vstep}{\L*sqrt(3)/(2*\d)} 	
		

		
		\begin{scope}[blend mode=darken]
			
			\fill[orange!55!white] (0,0) -- (\L,0) -- (\L - \L/2*1/4,{\L/2*sqrt(3)*1/4}) -- (\L/2*1/4,{\L/2*sqrt(3)*1/4});
			
		\end{scope}
		
		\draw[red,very thick] (\L*5/6,{\L/2*sqrt(3)*1/3}) -- (\L*1/6,{\L/2*sqrt(3)*1/3});
		

		
		\foreach \i in {0,...,\d}
		\pgfmathsetmacro{\k}{\d-\i}	
		\foreach \j in {0,...,\k}
		\fill (\hstep/2*\i + \hstep*\j,\vstep*\i) circle (0.8pt);
		

		
		\draw (0,0) -- (\L,0) -- (\L/2,{\L/2*sqrt(3)}) -- (0,0);
		

		\fill (0,0) node[left]{$x_0^d$};
		
		\fill (\L,0) node[right]{$x_1^d$};
		
		\fill (\L/2,{\L/2*sqrt(3)}) node[above]{$x_2^d$};

		\fill (\L/2*11/30,{\L/2*sqrt(3)*11/30}) node[rotate=60,above]{$m > \frac{ 2 d }{ 3 }$};
		
		\fill (\L/2,0) node[below]{$i_0+i_1 \geqslant m > \frac{ 2 d }{ 3 }$};

	\end{tikzpicture}
	\hspace{0.2cm}
	\begin{tikzpicture}[scale=0.8, font = \small]
		

		
		\pgfmathsetmacro{\L}{3}	
		\pgfmathsetmacro{\d}{13}	
		\pgfmathsetmacro{\hstep}{\L/\d}
		\pgfmathsetmacro{\vstep}{\L*sqrt(3)/(2*\d)} 	
		

		
		\begin{scope}[blend mode=darken]
			
			
			\fill[cyan!55!white] ({\hstep/2*(\d-10) + \hstep*1},{\vstep*(\d-8)}) -- ({\hstep/2*(\d-12) + \hstep*0},{\vstep*(\d-12)}) -- ({\hstep/2*3 + \hstep*(\d-6)},{\vstep*1}) -- ({\hstep/2*(\d-7) + \hstep*5},{\vstep*(\d-11)});
			
			\draw[orange!55!white,line width = 1.8pt] ({\hstep*(\d-4)},0) -- ({\hstep/2*3 + \hstep*(\d-7)},{\vstep*3});
			
		\end{scope}
		
		\draw[red,very thick] (\L/2*1/2,{\L/2*sqrt(3)*1/2}) -- (\L*1,0);
		

		
		\foreach \i in {0,...,\d}
		\pgfmathsetmacro{\k}{\d-\i}	
		\foreach \j in {0,...,\k}
		\fill (\hstep/2*\i + \hstep*\j,\vstep*\i) circle (0.8pt);
		

		
		\draw (0,0) -- (\L,0) -- (\L/2,{\L/2*sqrt(3)}) -- (0,0);
		

		\fill (0,0) node[left]{$x_0^d$};
		
		\fill (\L,0) node[right]{$x_1^d$};
		
		\fill (\L/2,{\L/2*sqrt(3)}) node[above]{$x_2^d$};

		\fill (\L/2,0) node[below]{$\frac{d}{3} \geqslant m > 0$};
		
		\fill (\L - \L/2*1/2,{\L/2*sqrt(3)*1/2}) node[rotate=-60,above]{$\frac{d-m}{2} < k \leqslant i_1$};

	\end{tikzpicture}
	\hspace{0.2cm}
	\begin{tikzpicture}[scale=0.8, font = \small]
		

		
		\pgfmathsetmacro{\L}{3}	
		\pgfmathsetmacro{\d}{13}	
		\pgfmathsetmacro{\hstep}{\L/\d}
		\pgfmathsetmacro{\vstep}{\L*sqrt(3)/(2*\d)} 	
		\pgfmathsetmacro{\x}{3}	
		\pgfmathsetmacro{\y}{1}	
		\pgfmathsetmacro{\z}{-4}	
		

		
		\begin{scope}[blend mode=darken]
			
			\fill[cyan!55!white] (0,0) -- (\L*3/8,0) -- (\L/2*3/8,{\L/2*sqrt(3)*3/8});
			
			\fill[orange!55!white] (\L/2*41/100,{\L/2*sqrt(3)*41/100}) -- (0,0) -- (\L,0) -- (\L - \L/2*4/25,{\L/2*sqrt(3)*4/25});
			
		\end{scope}
		
		\draw[red,very thick] ({\L/2*( 1-( \z/(\z-\x) ) )},{\L/2*sqrt(3)*( 1-( \z/(\z-\x) )}) -- ({\L - \L/2*( 1-( \z/(\z-\y) ) )},{\L/2*sqrt(3)*( 1 - \z/(\z-\y) )}); 
		

		
		\foreach \i in {0,...,\d}
		\pgfmathsetmacro{\k}{\d-\i}	
		\foreach \j in {0,...,\k}
		\fill (\hstep/2*\i + \hstep*\j,\vstep*\i) circle (0.8pt);
		

		
		\draw (0,0) -- (\L,0) -- (\L/2,{\L/2*sqrt(3)}) -- (0,0);
		

		\fill (0,0) node[left]{$x_0^d$};
		
		\fill (\L,0) node[right]{$x_1^d$};
		
		\fill (\L/2,{\L/2*sqrt(3)}) node[above]{$x_2^d$};

		\fill (\L/2*1/2,{\L/2*sqrt(3)*1/2}) node[rotate=60,above]{$i_0 \geqslant k > \frac{ m }{ 2 }$};
		
		\fill (\L - \L/2*1/2,{\L/2*sqrt(3)*1/2}) node[rotate=-60,above]{$\frac{d}{2} < m \leqslant \frac{2d}{3}$};
		
		\fill (\L*1/2,0) node[below]{$i_0+i_1 \geqslant m$};
		\fill (\L,0) node[below,white]{$\frac{d}{3}$};	

	\end{tikzpicture}
	\]
	\caption{The four cases of Proposition \ref{prop:plane_curves}. The colored area in any of the four triangles corresponds to nonzero coefficients of an example for $F$.} \label{fig:stability_picture}
\end{figure}

\begin{remark}
	Proposition \ref{prop:plane_curves} yields a practical algorithm which solves Problem \ref{prob:find_instability} for $n=2$. It has been implemented by the first named author, see \cite{KletusGitHub}.
	
	Here is a rough outline of the algorithm. Given a plane curve $X=V_+(F)\subset\PP^2_K$, it is easy to check whether  one of the Statements (a), (b), (c) or (d) holds. Also, there are at most finitely many lines $L$, points $P$, or pairs $(L,P)$ for which this is true.
	
	In each case where $L$ or $P$ as above exist, we choose a coordinate system $\E=(x_0,x_1,x_2)$ such that $L=V_+(x_0)$ resp.\ $P=[0:0:1]$, and write $F$ as in \eqref{eq:plane_curve1}. This determines the set $I_{F,\E}$. Then Condition \eqref{eq:plane_curve2} translates into a system of linear inequalities for $w_0,w_1,w_2\in\ZZ$. Using standard methods from linear optimization, it is easy to find an explicit solution, or to prove that no solution exists.
	
	So the algorithm either finds an explicit instability $(\E,w)$ for $X$, or it doesn't. In the latter case, the conclusion of Proposition \ref{prop:plane_curves} fails, which shows that $X$ is semistable.
\end{remark}

\begin{example} \label{exa:plane_curve}
	Let $K=\FF_3$, $d=5$ and
	\[
	F = x_0^3 x_2^2 + x_1^3 x_2^2 + x_1^5.
	\]
	Then $X=V_+(F)\subset\PP^2_K$ is an absolutely irreducible quintic, with a triple singularity at $P:=[0:0:1]$. In particular, $X$ does not contain a line, and hence we are neither in Case (a) nor in Case (c) of Proposition \ref{prop:plane_curves}. Since $d/2<m=3<2d/3$, we are also not in Case (b). The tangent cone of $X$ at $P$ is given by the equation
	\[
	x_0^3 + x_1^3 = ( x_0 + x_1 )^3,
	\]
	so it is the triple line $L:=V_+( x_0 + x_1 )$. Hence we may be in Case (d). Set $\E:=(y_0,y_1,y_2)$, with
	\[
	y_0:=x_0+x_1,\;\; y_1:=x_1,\;\; y_2:=x_2.
	\]
	With respect to the coordinate system $\E$, $P=[0:0:1]$ is still the triple singularity, but now $L=V_+(y_0)$. Furthermore,
	\[
	F = y_0^3y_2^2+y_1^5.
	\]
	One checks that $(\E,w)$, with
	\[
	w = (3,1,-4),
	\]
	is an instability for $X$, see the right hand side of \ref{fig:plane_curve_example}.
	
	In the above example, there is no instability for the standard coordinate system $(x_0,x_1,x_2)$, even though the point $P=[0:0:1]$ is a singularity of $X$ of multi\-pli\-city $>d/2$.
	This shows that the statements of Proposition \ref{prop:plane_curves} and of \cite[Corollary 7.4]{ES} are, although very similar, not equivalent.
\end{example}

\begin{figure}
	\[
	\begin{tikzpicture}
		

		
		\pgfmathsetmacro{\L}{3}	
		\pgfmathsetmacro{\d}{5}	
		\pgfmathsetmacro{\hstep}{\L/\d}
		\pgfmathsetmacro{\vstep}{\L*sqrt(3)/(2*\d)} 	
		\pgfmathsetmacro{\x}{3}	
		\pgfmathsetmacro{\y}{1}	
		\pgfmathsetmacro{\z}{-4}	
		

		
		\begin{scope}[blend mode=darken]
			
			\fill[cyan!55!white] (0,0) -- (\L*2/5,0) -- (\L/2*2/5,{\L/2*sqrt(3)*2/5});
			
			\fill[orange!55!white] (\L/2*2/5,{\L/2*sqrt(3)*2/5}) -- (0,0) -- (\L,0) -- (\L - \L/2*4/25,{\L/2*sqrt(3)*4/25});
			
		\end{scope}
		
		\draw[red,very thick] ({\L/2*( 1-( \z/(\z-\x) ) )},{\L/2*sqrt(3)*( 1-( \z/(\z-\x) )}) -- ({\L - \L/2*( 1-( \z/(\z-\y) ) )},{\L/2*sqrt(3)*( 1 - \z/(\z-\y) )}); 
		

		
		\foreach \i in {0,...,\d}
		\pgfmathsetmacro{\k}{\d-\i}	
		\foreach \j in {0,...,\k}
		\fill (\hstep/2*\i + \hstep*\j,\vstep*\i) circle (0.8pt);
		
		\fill (\hstep/2*2 + \hstep*0,\vstep*2) circle (1.5pt);
		\fill (\hstep/2*2 + \hstep*0,\vstep*2) node[left]{$x_0^3 x_2^2$};
		
		\fill (\hstep/2*2 + \hstep*3,\vstep*2) circle (1.5pt);
		\fill (\hstep/2*2 + \hstep*3,\vstep*2) node[right]{$x_1^3 x_2^2$};
		
		\fill (\hstep/2*0 + \hstep*5,\vstep*0) circle (1.5pt);
		

		
		\draw (0,0) -- (\L,0) -- (\L/2,{\L/2*sqrt(3)}) -- (0,0);
		

		\fill (0,0) node[left]{$x_0^5$};
		
		\fill (\L,0) node[right]{$x_1^5$};
		
		\fill (\L/2,{\L/2*sqrt(3)}) node[above]{$x_2^5$};

	\end{tikzpicture}
	\hspace{2cm}
	\begin{tikzpicture}
		

		
		\pgfmathsetmacro{\L}{3}	
		\pgfmathsetmacro{\d}{5}	
		\pgfmathsetmacro{\hstep}{\L/\d}
		\pgfmathsetmacro{\vstep}{\L*sqrt(3)/(2*\d)} 	
		\pgfmathsetmacro{\x}{3}	
		\pgfmathsetmacro{\y}{1}	
		\pgfmathsetmacro{\z}{-4}	
		

		
		\begin{scope}[blend mode=darken]
			
			\fill[cyan!55!white] (0,0) -- (\L*2/5,0) -- (\L/2*2/5,{\L/2*sqrt(3)*2/5});
			
			\fill[orange!55!white] (\L/2*2/5,{\L/2*sqrt(3)*2/5}) -- (0,0) -- (\L,0) -- (\L - \L/2*4/25,{\L/2*sqrt(3)*4/25});
			
		\end{scope}
		
		\draw[red,very thick] ({\L/2*( 1-( \z/(\z-\x) ) )},{\L/2*sqrt(3)*( 1-( \z/(\z-\x) )}) -- ({\L - \L/2*( 1-( \z/(\z-\y) ) )},{\L/2*sqrt(3)*( 1 - \z/(\z-\y) )}); 
		

		
		\foreach \i in {0,...,\d}
		\pgfmathsetmacro{\k}{\d-\i}	
		\foreach \j in {0,...,\k}
		\fill (\hstep/2*\i + \hstep*\j,\vstep*\i) circle (0.8pt);
		
		\fill (\hstep/2*2 + \hstep*0,\vstep*2) circle (1.5pt);
		\fill (\hstep/2*2 + \hstep*0,\vstep*2) node[left]{$y_0^3 y_2^2$};
		
		\fill (\hstep/2*0 + \hstep*5,\vstep*0) circle (1.5pt);
		

		
		\draw (0,0) -- (\L,0) -- (\L/2,{\L/2*sqrt(3)}) -- (0,0);
		

		\fill (0,0) node[left]{$y_0^5$};
		
		\fill (\L,0) node[right]{$y_1^5$};
		
		\fill (\L/2,{\L/2*sqrt(3)}) node[above]{$y_2^5$};

	\end{tikzpicture}
	\]
	\caption{Finding a semistability in Example \ref{exa:plane_curve}.} \label{fig:plane_curve_example}
\end{figure}
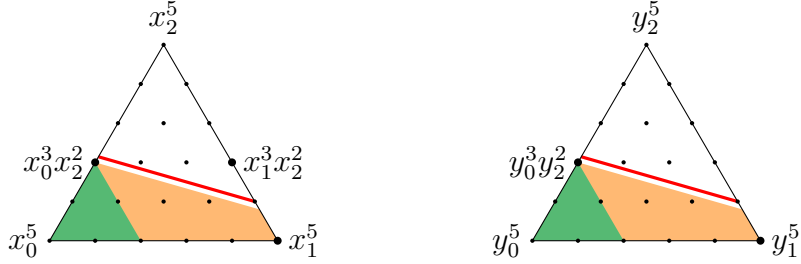

\subsection{The spherical complex} \label{subsec:spherical_complex}

As explained in \cite[\S 2.3]{MumfordGIT},
Problem \ref{prob:find_instability} can be translated into the problem of finding a point on the so-called \emph{spherical complex} $\Delta(G)$ where a certain function
\[
\nu(F,\,\cdot\,):\Delta(G)\to \RR,
\]
whose definition we recall below,
takes negative values. The \emph{stability function} we define in the next section, and which is the main actor of this article, is a natural generalization of the function $\nu(F,\,\cdot\,)$ to the case where the field $K$ is equipped with a nontrivial discrete valuation and the spherical complex $\Delta(G)$ is replaced by the Bruhat-Tits building of $G$. Although we do not use the complex $\Delta(G)$ and the function $\nu(F,\,\cdot\,)$ explicitly in the rest of this article, it is used implicitly in the proof of Proposition \ref{prop:stability_function_intro} (see in particular the proof of Lemma \ref{lem:stability_function4}). As we will see, it is also useful to prove the still-unproven Part (iv) of Proposition \ref{prop:plane_curves}.

\vspace{2ex}
Given a reductive algebraic group $G$,
the {\em spherical complex} $\Delta(G)$ is a certain simplicial complex whose simplices are the proper parabolic subgroups of $G$. In particular, the vertices correspond to the maximal parabolic subgroups. To each one-parameter subgroup $\lambda$ of $G$, we can associate a certain point $\delta=[\lambda]\in\Delta(G)$. Points of this type are called \emph{rational}. They form a dense subset of $\Delta(G)$ which contains all vertices. See \cite[\S 2.2]{MumfordGIT} or \cite{CLT}.

The general definition of $\delta=[\lambda]$ in \cite{CLT} is roughly the following. Let $T\subset G$ be a maximal split torus. It defines a subcomplex $A_T\subset \Delta(G)$ corresponding to the parabolic subgroups containing $T$, called an \emph{apartment}. One identifies the geometric realization of the complex $A_T$ with the set of all rays in the real vector space $X_*(T)\otimes\RR$. Then $\delta=[\lambda]$ is defined as the rays spanned by $\lambda\in X_*(T)$, where $T$ is some maximal split torus containing the image of $\lambda$.

Here we only consider the case $G=\SL_3(K)$. Then $\Delta(G)$ is actually a bipartite graph with the following very simple description, see e.g.\ \cite[\S 1]{everitt2014} or \cite[\S 6.5]{brown2008buildings}.
The vertices of $\Delta(G)$ of the first kind correspond to lines $L\subset\PP^2_K$, and the vertices of the second kind correspond to $K$-rational points $P\in\PP^2(K)$. Two vertices corresponding to a line $L$ and a point $P$ are joint by a unique edge if and only if $P\in L$.

Let $\E=(x_0,x_1,x_2)$ be a coordinate system and let $L_i\subset \PP^2_K$ be the line $x_i=0$, for $i=0,1,2$. Then the six vertices corresponding to the three lines $L_0,L_1,L_2$ and their points of intersection $P_0,P_1,P_2$ span a circular subgraph $A_\E$ of $\Delta(G)$, called the {\em spherical apartment} associated to $\E$, see Figure \ref{fig:spherical_apartment}.  It is easy to see that  $A_\E$ only depends on the maximal split torus of $G$ defined by $\E$. Also, the building $\Delta(G)$ is the union of all these apartments $A_\E$.

Let $\lambda$ be a one-parameter subgroup corresponding to a pair $(\E,w)$, as in \eqref{eq:ops_lambda}. Then the rational point $\delta=[\lambda]\in\Delta(G)$ lies on the apartment $A_\E$ and is determined as follows. If $w_1=w_2$ then $\delta$ is completely determined by the (degenerate) flag
\[
\{0\} \subsetneq \gen{x_0} \subsetneq V^*.
\]
Indeed, the stabilizer of this flag is the unique maximal parabolic subgroup of $G$ containing the image of $\lambda$. In terms of the above description of $\Delta(G)$ as a bipartite graph, this means that $\delta$ is the vertex corresponding to the line $L_0=V_+(x_0)$. Similarly, if $w_0=w_1$, then $\delta$ gives rise to the (degenerate) flag
\[
\{0\} \subsetneq \gen{x_0,x_1} \subsetneq V^*,
\]
and hence it is the vertex corresponding to the point $P_1=V_+(x_0,x_1)$.
If $w_0>w_1>w_2$, then $\delta$ gives rise to the full flag
\[
\{0\} \subsetneq \gen{x_0}\subsetneq\gen{x_0,x_1} \subsetneq V^*
\]
of $V^*$. The stabilizer of this flag is the parabolic subgroup of $G$ corresponding to the edge $e_0$ of $\Delta(G)$ connecting the vertices corresponding to $L_0$ and $P_1$. The point $\delta=[\lambda]$ lies in the interior of this edge. In fact, using a suitable parametrization $e_0\cong[0,1]$ of the edge $e_0$ we have
\begin{equation} \label{eq:edge_parametrization}
	\delta=[(\E,w)] =\frac{w_1-w_2}{w_0-w_2}.
\end{equation}

\begin{figure}[h]
	\begin{center}
		\begin{tikzpicture}[scale=1.0]
			
			\newcommand{\radius}{2}
			
			\coordinate (V0) at (0:\radius);    
			\coordinate (V1) at (60:\radius);   
			\coordinate (V2) at (120:\radius);  
			\coordinate (V3) at (180:\radius);  
			\coordinate (V4) at (240:\radius);  
			\coordinate (V5) at (300:\radius);  
			\coordinate (E)  at (27:1.15*\radius);
			
			\newcommand{\drawVertex}[3]{%
				\node[draw=black,
				fill=#3,
				circle,
				inner sep=2pt,
				label={#2}] at (#1) {};
			}
			
			\draw (0,0) circle (\radius);
			
			\drawVertex{V0}{right:{$L_0$}}{black}
			\drawVertex{V1}{above right:{$P_1$}}{white}
			\drawVertex{V2}{above left:{$L_1$}}{black}
			\drawVertex{V3}{left:{$P_0$}}{white}
			\drawVertex{V4}{below left:{$L_2$}}{black}
			\drawVertex{V5}{below right:{$P_2$}}{white}

			\node at (E) {$e_0$};
			
		\end{tikzpicture}
	\end{center}
	\caption{The spherical apartment $A_\E$.}\label{fig:spherical_apartment}
\end{figure}
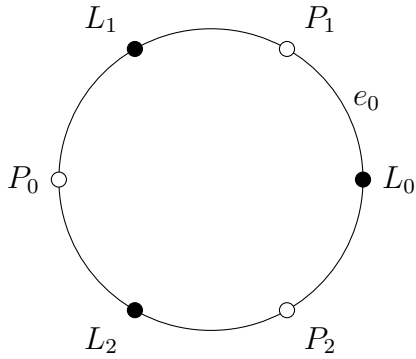

Given this explicit description of the building $\Delta(G)$, it is easy to check that
\begin{equation} \label{eq:nu}
	\nu(F,[(\E,w])) := -\frac{1}{\norm{w}}\, \min_{\;i\in I_{F,\E}} \gen{i,w},
\end{equation}
defines a continuous function $\nu(F,\,\cdot\,):\Delta(G)\to\RR$. Theorem \ref{thm:numerical_criterion}  states that $X$ is semistable (resp.\ stable) if and only of the function $\nu(F,\,\cdot\,)$ only takes nonnegative (resp.\ only positive) values.

\vspace{2ex}
We are now ready to give a proof of the remaining Part (iv) of Proposition \ref{prop:plane_curves}.

\begin{proof} \label{page:proof_of_lemma}
	Let $n=2$ and $X=V_+(F)\subset \PP(V)$ be a plane curve of degree $d\geq 3$. We assume that $X$ is unstable, and that Statements (a) and (b) of Proposition \ref{prop:plane_curves} (i) are both false.

	By Theorem \ref{thm:numerical_criterion}, the subset
	\[
	C_F := \{ \delta\in\Delta(G) \mid \nu(F,\delta) < 0 \}
	\]
	of $\Delta(G)$ is nonempty. We claim that $C_F$ does not contain any vertex of $\Delta(G)$. To see this, assume first that $C_F$ contains a vertex corresponding to a line $L\subset\PP(V)$. In the above description of the points of $\Delta(G)$, this vertex is represented by an instability $(\E,w)$ with $w_1=w_2$. By the proof of Proposition \ref{prop:plane_curves} (i) this shows that Statement (a) holds for the line $L:=V_+(x_0)$, contrary to our assumption. A very similar argument, using our exclusion of Statement (b), shows that $C_F$ also cannot contain vertices corresponding to $K$-rational points of $\PP(V)$. This concludes the proof of the claim.
	
	By \cite[Corollary 2.16]{MumfordGIT}, $C_F$ is \emph{convex}, in the sense of  \cite[Definition 2.10]{MumfordGIT}. In our case this means that the intersection of $C_F$ with any apartment $A_\E$ is an open interval which does not contain any pair of opposite points. Since any two points of $\Delta(G)$ are contained in a common apartment, this also implies that $C_F$ is connected.
	As we have shown above, $C_F$ does not contain any vertices. We conclude that $C_F$ is an open interval contained in the interior of a unique edge of $\Delta(G)$.
	
	Let $e$ be an arbitrary edge of $\Delta(G)$, and let $(L,P)$ be the pair corresponding to $e$. Choose any coordinate system $\E=(x_0,x_1,x_2)$ such that $L=V_+(x_0)$ and $\{P\}=V_+(x_0,x_1)$. Then each rational point $\delta$ in the interior of $e$ is represented by $(\E,w)$, where $w$ is a balanced and and ordered weight vector with $w_0>w_1>w_2$. Furthermore, $\delta\in C_F$ if and only if $(\E,w)$ is an instability for $X$. And if this is the case, then the proof of Part (i) of Proposition \ref{prop:plane_curves} shows that one of the Statements (c) or (d) holds for $L$ and $P$. We see that Proposition \ref{prop:plane_curves} (iv) is equivalent to the claim that the subset $C_F$ is contained in the interior of a unique edge $e$ of $\Delta(G)$, and this is exactly what we have shown above.
\end{proof}

\subsection{Semi-instabilities} \label{subsec:semi-instabilities}

Assume that $X$ is semistable. A natural question, similar to Problem \ref{prob:find_instability}, is to (a) decide whether or not $X$ is stable, and (b) if $X$ is semistable but not stable, find one (or all) pairs $(\E,w)$ such that
\begin{equation} \label{eq:semi-instabilities1}
	\forall\,i\in I_{F,\E}:\;\gen{i,w} \geq 0,
\end{equation}
see Corollary \ref{cor:instability}). For $n=3$ it is possible to prove a result similar to Proposition \ref{prop:plane_curves} which gives a full answer to both (a) and (b). We will come back to this question in a subsequent article.

Let us call a pair $(\E,w)$, or the corresponding one-parameter subgroup $\lambda$, a \emph{semi-instability} for $X$ if \eqref{eq:semi-instabilities1} holds. Just as for instabilities, we may consider the subset
\[
\bar{C}_F := \{ \delta \in\Delta(G) \mid \nu(F,\delta)\leq 0\}.
\]
The assumption that $X$ is semistable but not stable means that $\bar{C}_F$ is precisely the set of zeroes of the function $\nu(F,\,\cdot\,)$. By \cite[Corollary 2.16]{MumfordGIT}, the set $\bar{C}_F$ is \emph{semi-convex}, in the sense of \cite[Definition 2.10]{MumfordGIT}. The following example shows that this result is a lot weaker than the convexity of $C_F$, which was an important ingredient in the proof of Proposition \ref{prop:plane_curves} (iv). For instance, $\bar{C}_F$ is in general not connected, and may even have infinitely many connected components.

\begin{example} \label{exa:semi-instabilities}
	Let $K$ be an infinite field, and let $C=V_+(Q)\subset\PP^2_K$ be a smooth plane conic with at least one $K$-rational point. We consider the plane quartic $X:=2\cdot C=V_+(Q^2)$. Then $X$ does not contain any line, and all of its $K$-rational points have multiplicity $2$. It follows from Proposition \ref{prop:plane_curves} that $X$ is semistable. We will show that $X$ is not stable and that the subset $\bar{C}_F\subset \Delta(G)$ where the function $\nu(F,\,\cdot\,)$ achieves its minimal value zero contains infinitely many isolated points and is not contained in a finite number of apartments.
	
	\begin{figure}[h]
		\centering
		\begin{tikzpicture}[scale=1]
			
			\coordinate (O) at (0,0);
			
			\newcommand{\radius}{2}
			
			\draw[name path=C] (O) circle (\radius);
			
			\pgfmathsetmacro{\anglePone}{30}
			\pgfmathsetmacro{\anglePtwo}{150}
			
			\coordinate (P1) at (\anglePone:\radius);
			\coordinate (P2) at (\anglePtwo:\radius);
			
			\node[fill=black,circle,inner sep=1pt,label={above right:$P_1$}] at (P1) {};
			\node[fill=black,circle,inner sep=1pt,label={above left:$P_2$}] at (P2) {};
			
			\draw[name path=L0, thick]
			($(P1)!-0.5!(P2)$) -- ($(P2)!-0.5!(P1)$) node[midway, above]{$L_0$};

			\path let \p1 = (P1) in coordinate (P1Normal) at (-\y1,\x1);
			\path let \p2 = (P2) in coordinate (P2Normal) at (-\y2,\x2);
			
			\draw[name path=L1, thick]
			($(P1)-1*(P1Normal)$) -- ($(P1)+2.5*(P1Normal)$);
			\draw[name path=L2, thick]
			($(P2)-2.5*(P2Normal)$) -- ($(P2)+1*(P2Normal)$);
			
			\path[name intersections={of=L1 and L2,by=P_0}];
			
			\node at ($(P1)!0.5!(P1)+(P1Normal)$) [above right] {$L_1$};
			\node at ($(P2)-(P2Normal)$) [above left] {$L_2$};
			
			\node[fill=black,circle,inner sep=1pt,label={left:$P_0$}] at (P_0) {};

		\end{tikzpicture}
		\caption{The three lines $L_0,L_1,L_2$ defined by two points on a conic.}\label{fig:double_conic}
	\end{figure}
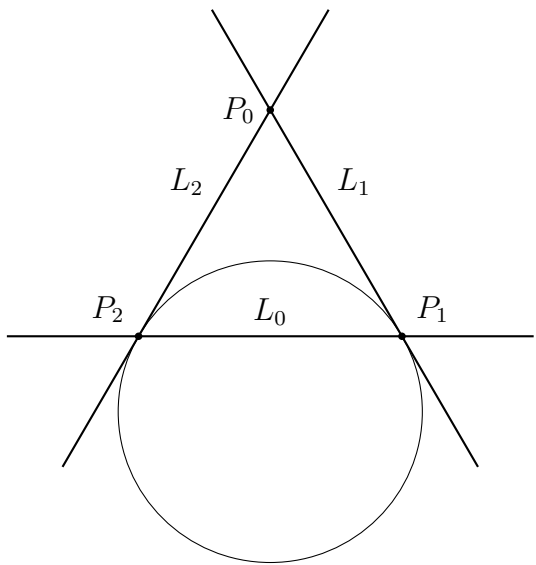

	Let $P_1,P_2\in C(K)$ be any pair of distinct $K$-rational points. Let $L_i$ be the tangent line to $C$ through $P_i$, and let $L_0$ be the line through $P_1$ and $P_2$. As the three lines are in general position, we may choose a coordinate system $\E=(x_0,x_1,x_2)$ such that $L_i$ is the line $x_i=0$, for $i=0,1,2$. See Figure \ref{fig:double_conic}. It is then easy to check that, after multiplying the $x_i$ by suitable constants, the equation for the conic $C$ is
	\[
	Q = x_0^2 + x_1x_2.
	\]
	So the equation for $X$ is
	\begin{equation} \label{eq:double_conic}
		F:=Q^2=x_0^4 + 2x_0^2x_1x_2 + x_1^2x_2^2,
	\end{equation}
	and
	\[
	\{ (4,0,0), (0,2,2)\}\subset I_{F,\E} \subset \{(4,0,0),(0,2,2),(2,1,1)\},
	\]
	depending on whether $K$ has odd or even characteristic.
	Clearly, \eqref{eq:semi-instabilities1} holds if and only if  $w=(0,m,-m)$, for some $m\in\ZZ$. In particular, the two pairs
	\begin{equation} \label{eq:semi-instabilites}
		((x_1,x_0,x_2), (1,0,-1)), \quad (x_2,x_0,x_1),(1,0,-1))
	\end{equation}
	are semi-instabilities. It follows that $X$ is not stable.
	
	Let $\delta_1,\delta_2\in\Delta(G)$ denote the two points corresponding to the pairs \eqref{eq:semi-instabilities1}. These are precisely the zeroes of the function $\nu(F,\,\cdot\,):\Delta(G)\to\RR$ on the apartment $A_\E$. Note that this does not contradict the claim that $\bar{C}_F$ is semi-convex, since $\delta_1$ and $\delta_2$ are \emph{antipodal} on the apartment $A_\E$ (\cite[Definition 2.8]{MumfordGIT}).
	
	The pair of zeroes $\delta_1,\delta_2$ depends on the choice of the pair of points $P_1,P_2\in X$. Conversely, the pair $\delta_1,\delta_2$ uniquely determine the pair $\delta_1,\delta_2$. To see this, note that $\delta_1$ lies in the interior of the edge connecting the vertices corresponding to the line $L_1$ and the point $P_1$. Similarly,  $\delta_2$ lies in the interior of the edge connecting the vertices corresponding to the line $L_2$ and the point $P_2$. Since $K$ is an infinite field, there are infinitely many pairs of points $P_1,P_2$, yielding an infinite set of isolated zeroes of the function $\nu(F,\,\cdot\,)$ on $\Delta(G)$.
	
	This example will be used in \S \ref{subsec:example2} to construct a smooth quartic over a discretely valued field with infinitely many nonisomorphic semistable models.
\end{example}

\section{Valuations and buildings}

In this section we define the Bruhat-Tits building $\B_K$ mentioned in the introduction as a space of homothety classes of valuations on the $K$-vector space $V^*$, following the work of Goldman and Iwahori (\cite{GoldmanIwahori}).

All this is well known (see e.g.\ \cite[\S 3]{ASENS_2011}, \cite[\S 2.2]{RTW15}, \cite{parreau_immeubles}, \cite[\S 1]{BKKUV}), so the purpose of this section is to fix our notation and to highlight the properties of $\B_K$ which are most relevant for us. We adopt the following conventions and assumptions, which differ from those in some of the articles referenced above:
\begin{itemize}
	\item
	We systematically use (additive) valuations instead of (multiplicative) norms.
	\item
	We assume that $K$ is discretely valued; this implies that it is spherically complete (see e.g.\ \cite[\S 2.4.4]{BGR}). As a result, we do not have to distinguish between diagonalizable and non-diagonalizable valuations (resp.\ norms). This considerably simplifies our notation relative to the reference given above.
	\item
	On the other hand, we do not assume that $K$ is a local field. As a result, the building $\B_K$ is in general not locally compact (with respect to the metric topology), making it harder to prove finiteness results.
	\item
	The projective space $\PP(V)$ is defined as the projective scheme ${\rm Proj}\, K[V]$, so that $H^0(\PP(V),\OO(1))=V^*$. As a result, we will mostly work with the vector space $V^*$, and not $V$. This is different from the convention used in \cite{ASENS_2011}, but agrees with \cite{BKKUV}.
\end{itemize}

\subsection{Spaces of valuations} \label{subsec:valuations}

Let $(K, v_K)$ be  a  complete discretely valued field. We write $\OO_K$ for the ring of integers of $K$,  $\pi=\pi_K$ for a uniformizing element, $k$ for its residue field and $\Gamma_K:=v_K(K^\times)\subset\RR$ for the value group.

We also fix a $K$-vector space $V$ of dimension $n+1$. We let $\N_K$ denote the set of all {\em $v_K$-valuations} on $V^*$, i.e.\ the set of maps
\[
v:V^*\to\RRb
\]
such that for all $x,y\in V^*$ and $c\in K$ we have
\begin{equation}
	\begin{split}
		v(x+y) \geq \min(v(x),v(y)), \\
		v(c\cdot x) = v_K(c) + v(x), \\
		v(x) = \infty \;\;\Leftrightarrow\;\; x=0.
	\end{split}
\end{equation}
Two valuations $v,v'\in\N_K$ are said to be {\em homothetic}, $v\sim v'$, if $v(x)-v'(x)$ is a constant, independent of $x\in V^*$. We set
\[
\B_K := \N_K/_\sim.
\]
We call $\N_K$ the \emph{Goldman-Iwahori space} and $\B_K$ the \emph{Bruhat-Tits building} of the $K$-vector space $V^*$.

For $v\in\N_K$ and $g\in\GL(V)$ we set
\[
\lexp{g}{v}:V^*\to\RRb, \quad x\mapsto v(\lexp{g^{-1}}{x}) = v(x\circ g)
\]
(note that we use the convention of treating $V^*$ as the contragrediant of $V$ as the natural representation of $\GL(V)$, as in \S \ref{subsec:stability_definitions}). This defines a natural left action of $\GL(V)$ on $\N_K$, which descents to an action of $\PGL(V)$ on $\B_K$.

\begin{definition} \label{def:vertices}
	A valuation $v\in\N_K$ is called \emph{integral} if $v(V^*\backslash\{0\})=\Gamma_K$.
	We denote the subset of $\N_K$ corresponding to integral valuations by $\N_K^\circ$, and its image in $\B_K$ by $\B_K^\circ$. A point $b=[v]\in\B_K^\circ$ is called a \emph{vertex}.
\end{definition}

\begin{remark}	\label{rem:vertices_and_lattices}
	If $v\in\N_K^\circ$ is integral, then
	\[
	L:=\{x\in V^* \mid v(x)\geq 0\}
	\]
	is a full $\OO_K$-lattice of the $K$-vector space $V^*$. We obtain a natural identification of the vertex set $\B_K^\circ$ with the set of homothety classes of lattices $L\subset V^*$.
\end{remark}

As recalled below (Remark \ref{rem:building} (v)), the set $\B_K$ as defined here, is the geometric realization of the simplicial complex with vertex set $\B_K^\circ$ defined in \S \ref{subsec:stability_function_intro}.

\subsection{$\B_K$ as an affine building}

We fix an integer $n\geq 1$ and set
\[
\EE_n := \RR^{n+1}/\RR\cdot(1,\ldots,1).
\]
We equip $\EE_n$ with the structure of euclidean vector space via the isomorphism
\[
(\RR\cdot(1,\ldots,1))^\perp = \{w=(w_0,\ldots,w_n)\mid w_0+\ldots+w_n=0\} \iso\EE_n.
\]

Let $\Lambda_K\subset\EE_n$ denote the image of $\Gamma_K^{n+1}\subset\RR^{n+1}$, which is a full lattice. Let $\overline{W}\subset\GL(\EE_n)$ be the group of linear isometries of $\EE_n$ generated by the automorphisms which are induced by a permutation of the coordinates. Then
\[
W_K := \Gamma_K\rtimes\overline{W}
\]
is a finite reflexion group, and the pair $(\EE_n,W_K)$ is an \emph{affine Coxeter complex, of type $A_n$}. In particular, $\EE_n$ is the geometric realization of a certain simplicial complex for which the lattice $\Lambda_K$ is the set of vertices. See e.g\ \cite{brown2008buildings}.

\begin{definition} \label{exa:diagonalized_valuation}
	Let $\E=(x_0,\ldots,x_n)$ be a basis of $V^*$, and $w=(w_0,\ldots,w_n)\in \RR^{n+1}$. We denote by $v_{\E,w}$ the valuation
	\[
	v_{\E,w}:	V^*\to\RR\cup\{\infty\}, \quad
	x=\sum_i a_ix_i \mapsto \min\{ v_K(a_i)+w_i \mid i=0,\ldots,n\}.
	\]
	We say that a valuation $v$ on $V^*$ is \emph{diagonalized} by the basis $\E$ if $v=v_{\E,w}$ for some $w\in\RR^{n+1}$.
\end{definition}

We recall the following well known fact.

\begin{proposition} \label{prop:v_diagonalize}
	Let $v,v'$ be two valuations on $V^*$. Then there exists a basis of $V^*$ which simultaneously diagonalizes $v$ and $v'$.
\end{proposition}

\begin{proof}
	This is a special case of \cite[Cor. 3.11]{parreau_immeubles}. In case that $K$ is a local field, a more direct proof is given in \cite[Proposition 2.21]{RTW15}.
\end{proof}

For any basis $\E$ of $V^*$ we obtain an injective map
\[
\tilde{f}_\E:\RR^{n+1} \hookrightarrow \N_K, \quad w\mapsto v_{\E,w},
\]
and one easily checks that it induces an injective map
\[
f_\E:\EE_n\hookrightarrow \B_K, \quad [w]\mapsto [v_{\E,w}]
\]
The images
\[
\A_\E :=f_\E(\EE_n)\subset\B_K
\]
are called the {\em apartments} of $\B_K$. The map $f_\E$ is called a \emph{marked apartment}. By Proposition \ref{prop:v_diagonalize}, any two points on $\B_K$ lie on a common apartment. In particular, $\B_K$ is the union of all its apartments.

\begin{remark} \label{rem:building}
	\begin{enumerate}[(i)]
		\item
		The set $\B_K$, equipped with the family of marked apartments $f_\E$, is an {\em affine building of type $(\EE_n,W_K)$}, i.e.\ it satisfies the axioms (A 1-4) and (A 5') from \cite[\S 1.2]{parreau_immeubles}. This is proved in \cite[\S 3]{parreau_immeubles}.
		\item
		It follows from (i) and \cite[Proposition 1.3]{parreau_immeubles} that there is a unique metric $d:\B_K\times\B_K\to\RR_{\geq 0}$ on $\B_K$ which restricts to the euclidean metric on each apartment induced by the isomorphism $f_\E:\EE_n\iso\A_\E$. We will henceforth consider $\B_K$ as a metric space with this metric.
		\item
		If $K$ is a local field, then $\B_K$ is locally compact with respect to the topology induced by the metric $d$. But if the residue field $k$ is not finite, this is not true anymore.
		\item
		It is clear from the definitions that
		\[
		f_\E(\Lambda_K) = \A_\E\cap\B_K^\circ.
		\]
		So via $f_\E$ the vertices of the simplicial complex $(\EE_n,W_K)$ correspond to the vertices of $\B_K$, as defined in Definition \ref{def:vertices}. Explicitly, if $[v]\in\B_K^\circ$ is a vertex, then
		\[
		L_v:=\{x\in V^*\mid v(x) \geq 0\}
		\]
		is a full $\OO_K$-lattice of $V^*$. Its homothety class only depends on the point $[v]$.
		\item
		The simplicial structure on $(\EE_n,W_K)$ induces, via the marked apartments $f_\E$, a simplicial structure on $\B_K$. By (iv) the vertices of this simplicial structure are precisely the vertices of $\B_K$ from Definition \ref{def:vertices}. It is then easy to verify (see e.g.\ \cite[\S V.8]{brown2008buildings} or \cite[\S 2.2]{RTW15}) that we recover the simplicial complex with vertex set $\B_K^\circ$ defined in \S \ref{subsec:stability_function_intro}.
	\end{enumerate}
\end{remark}

\subsection{Extending the base field} \label{subsec:B_K_in_B_L}

Let $L/K$ be a finite field extension. Then the valuation $v_K$ extends to a unique, discrete valuation $v_L$ on $L$ with respect to which $L$ is also complete. We obtain an inclusion $\Gamma_K\subset\Gamma_L$ of discrete subgroups of $\RR$.

\begin{proposition}
	Let $L/K$ be a finite field extension and $v\in\N_K$. Consider the map $\tilde{v}:V_L^*\to\RRb$ defined by
	\[
	\tilde{v}(x) = \sup\; \min_i (v_L(a_i)+v(x_i)),
	\]
	where the supremum is defined over all representations $x=\sum_i a_ix_i$ with $a_i\in L$ and $x_i\in V$.
	Then $\tilde{v}\in\N_L$ and $\tilde{v}|_{V^*}=v$. Moreover, if $\E=(x_0,\ldots,x_n)$ is a $K$-basis of $V^*$ diagonalizing $v$, then we have
	\[
	\tilde{v}(a_0x_0+\ldots+a_nx_n) = \big(\min_i v_L(a_i) + v(x_i)\big),
	\]
	for all $(a_i)\in L^{n+1}$.
\end{proposition}

\begin{proof}
	This is clear.
\end{proof}

\begin{remark}
	The proposition defines an injection $\N_K\hookrightarrow \N_L$, $v\mapsto \tilde{v}$, which descends to an injection
	\begin{equation} \label{eq:B_K_to_B_L}
		\iota:\B_K\hookrightarrow \B_L.
	\end{equation}
	It has the following properties.
	\begin{enumerate}[(i)]
		\item
		It is equivariant with respect to the left actions of the groups $\GL(V)$ and $\GL(V_L)$.
		\item
		Let $\E$ be  a $K$-basis of $V^*$, and let $\E_L$ denote $\E$ considered as an $L$-basis of $V_L^*$. Then
		\[
		\iota\circ f_\E = f_{\E_L}.
		\]
		\item
		It follows from (ii) that the map $\iota$ is a metric embedding.
		\item
		It induces an injective map $\B_K^\circ\hookrightarrow\B_L^\circ$; on each apartment $\A\subset\B_K$, it corresponds to the natural inclusion of lattices
		\[
		\Lambda_K\cong\A\cap\B_K\hookrightarrow\iota(\A)\cap\B_L\cong\Lambda_L.
		\]
		\item 
		If $L/K$ is a Galois extension then $\Gal(L/K)$ acts on $\B_L$ via simplicial isometries, and this action commutes with the embedding $\iota$.
	\end{enumerate}
	All this justifies to consider $\B_K$ as a `sub-building' of $\B_L$ via $\iota$, with the caveat of (iv): with respect to the simplicial structure, $\iota(\B_K)$ is a refinement of $\B_K$.
\end{remark}

\section{The stability function} \label{sec:stability_function}

We continue with the notation and assumptions of the previous section. So $K$ is a complete discretely valued field and $V$ is a $K$-vector space of dimension $n+1$. To $V$ we attach the Goldman-Iwahori space $\N_K$ and its quotient space, the Bruhat-Tits building $\B_K$.

We also consider a homogenous form $F\in K[V]_d$ of degree $d\geq 3$, which defines a hypersurface $X:=V_+(F)\subset\PP(V)$. We assume that $X$ is stable, in the sense of Definition \ref{def:stability}. Our goal is to define and study a certain map
\[
\phi_X:\B_K\to \RR
\]
called the \emph{stability function} of $F$, see \S \ref{subsec:stability_function_intro}. In particular, we give a proof of Proposition \ref{prop:stability_function_intro}.

\subsection{Extension to a multiplicative valuation}
\label{subsec:embedding_into_P(V)^an}

A valuation $v\in\N_K$, which is a map $v:V^*\to\RRb$, has a natural extension to a valuation $v:K[V]\to\RRb$ on the \emph{ring} $K[V]$. In particular, for $F_1,F_2\in K[V]$ we have $v(F_1F_2)=v(F_1)+v(F_2)$. See e.g.\ \cite[Proposition 3.2]{ASENS_2011}. We recall this construction.

We choose a basis $\E=(x_0,\ldots,x_n)$ of $V^*$ which diagonalizes $v$, and define the extension of $v$ by
\begin{equation} \label{eq:v_tilde}
	v(\sum_i a_ix_0^{i_0}\cdots x_n^{i_n}) := \min_i \big(v_K(a_i) + i_0v(x_0)+\ldots+i_nv(x_n) \big).
\end{equation}
This is clearly a valuation on the ring $K[V]$ extending the original valuation $v$ which was only defined on $V^*\subset K[V]$. In fact, it is easy to see that it is the minimal extension of $v$ to a valuation on the ring $K[V]$. This shows that our definition does not depend on the choice of the basis $\E$. It is also clear that if $v_1,v_2\in \N_K$ are homothetic, then their extensions are homothetic, too. To be precise, if $v_1(x)=v_2(x)+c$ for all $x\in V^*$, then
\begin{equation} \label{eq:v_tilde2}
	v_1(F) = v_2(F) + d\cdot c
\end{equation}
for a form $F\in K[V]_d$ of degree $d$. For $g\in\GL(V)$, our conventions introduced in \S \ref{subsec:stability_definitions} and \S \ref{subsec:valuations} mean that
\begin{equation} \label{eq:v_tilde_g}
	(\lexp{g}{v})(F) = v(\lexp{g^{-1}}{F}).
\end{equation}

\subsection{Definition of the stability function}

We fix a $K$-basis $\E_0$ of $V^*$, and write $v_0:=v_{\E_0,0}\in\N_K$. Any $K$-basis $\E$ of $V^*$ is of the form $\E=g(\E_0)$, for a unique $g\in\GL(V)$. We set
\[
\det(\E) := \det(g) \in K^\times.
\]

We define a function
\begin{equation} \label{eq:omega_def}
	\omega:\N_K\to\RR, \quad \omega(v) := \max_{\E} \frac{1}{n+1}\Big(v(x_0)+\ldots+v(x_n) -v_K(\det \E)\Big),
\end{equation}
where $\E=(x_0,\ldots,x_n)$ runs over all $K$-bases of $V^*$.

\begin{lemma} \label{lem:omega}
	\begin{enumerate}[(i)]
		\item
		Let $\E=(x_0,\ldots,x_n)$ be a $K$-basis, $w=(w_0,\ldots,w_n)\in\RR^{n+1}$ and $v:=v_{\E,w}$. Then
		\[
		\omega(v) = \frac{1}{n+1}\big(w_0+\ldots+w_n -v_K(\det \E)\big),
		\]
		In other words: the maximum in \eqref{eq:omega_def} is attained by any basis $\E$ diagonalizing $v$.
		\item
		We have $\omega(v_0)=0$.
		\item
		For $v\in\N_K$ and $g\in GL(V)$ we have
		\[
		\omega(\lexp{g}{v}) = \omega(v) - \frac{v_K(\det g)}{n+1}.
		\]
		\item
		If $v\sim v'$ are homothetic, i.e.\ $v'(x)=v(x)+c$ for a constant $c\in\RR$, then
		\[
		\omega(v')=\omega(v) + c.
		\]
	\end{enumerate}
\end{lemma}

\begin{proof}
	Let $\E=(x_0,\ldots,x_n),w,v$ be as in (i), and let $\E'=g(\E)$ be another basis, $g\in\GL(V)$. Then
	\begin{equation} \label{eq:omega1}
		v_K(\det\E') = v_K(\det\E) + v_K(\det g).
	\end{equation}
	Moreover, by \cite[Lemma 3.2]{parreau_immeubles} we have
	\begin{equation} \label{eq:omega2}
		v(\lexp{g}{x}_0)+\ldots+v(\lexp{g}{x}_n) \leq
		w_0+\ldots+w_n + v_K(\det g).
	\end{equation}
	Combining \eqref{eq:omega1} and \eqref{eq:omega2} shows Part (i) of the lemma. Parts (ii)-(iv) follow immediately.
\end{proof}

Let $F\in K[V]_d$ be a form of degree $d$. Up until the proof of Lemma \ref{lem:stability_function2} we do not need to make any stability assumptions. We define a map
\begin{equation} \label{eq:phi_X_def}
	\phi_X:\B_K\to\RR, \quad \phi_X([v]) := d\cdot \omega(v) - v(F),
\end{equation}
where $v(F)$ is defined by \eqref{eq:v_tilde}.
Lemma \ref{lem:omega} (iii) and \eqref{eq:v_tilde2} show that $\phi_X$ is well defined, i.e.\ the right hand side of \eqref{eq:phi_X_def} only depends on the homothety class of $v$.

\begin{definition}
	The map $\phi_X:\B_K\to\RR$ defined by \eqref{eq:phi_X_def} is called the {\em stability function} of the form $F$ (or of the hypersurface $X=V_+(F)$ defined by $F$).
\end{definition}

\begin{remark}
	For a nonzero constant $c\in K^\times$ and any $v\in\N_K$ we have
	\[
	\phi_{cF}([v]) = \phi_X([v]) - v_K(c).
	\]
	None of the properties of $\phi_X$ we are interested in is affected by adding a constant. We may therefore assume that $F$ is normalized such that $v_0(F)=0$. This means that $F$ has integral coefficients and is primitive when written as a polynomial in our standard basis $\E_0$. By Lemma \ref{lem:omega} (ii) we then have
	\begin{equation} \label{eq:phi_X_normalized}
		\phi_X([v_0])=0.
	\end{equation}
\end{remark}

\subsection{Proof of Proposition \ref{prop:stability_function_intro}} \label{subsec:proof_of_prop_1.5}

We are now going to prove that the function $\phi_X$ defined above has all the properties claimed in Proposition \ref{prop:stability_function_intro}. The proof is divided into four steps, Lemma \ref{lem:stability_function1}, \ref{lem:stability_function2}, \ref{lem:stability_function3} and \ref{lem:stability_function4}. The assumption that $X=V_+(F)$ is stable is only used from Lemma \ref{lem:stability_function2} on.

Let us first fix some notation. We set
\[
I_d := \{i=(i_0,\ldots,i_n) \mid i_j\in\NN_0,\; i_0+\ldots+i_n=d\}.
\]
For each $i\in I_d$ we consider the linear form $l_i:\RR^{n+1}\to\RR$ defined by
\begin{equation} \label{eq:linear_form_l_i}
	l_i(w):= \frac{d}{n+1}(w_0+\ldots+w_n) - (i_0w_0+\ldots +i_nw_n).
\end{equation}
By definition, it vanishes on the subspace $\RR\cdot(1,\ldots,1)$, and hence may be regarded as a linear form on $\EE_n$.

The following lemma proves that $\phi_X$ is uniformly continuous and most of Part (i) of Proposition \ref{prop:stability_function_intro}.

\begin{lemma} \label{lem:stability_function1}
	\begin{enumerate}[(i)]
		\item
		Let $\E=(x_0,\ldots,x_n)$ be a basis of $V^*$ and let $\phi_{F,\E}:=\phi_X\circ f_\E:\EE_n\to\RR$ denote the restriction of $\phi_X$ to the marked apartment $f_\E$. Write
		\[
		F = \sum_{i\in I_{F,\E}} a_i x_0^{i_0}\cdots x_n^{i_n},
		\]
		where $I_{F,\E}\subset I_d$ is the subset for which $a_i\neq 0$.
		Then
		\[
		\phi_{F,\E}(w) = \max_{i\in I_{F,\E}}(l_i(w)-v_K(a_i)) -\frac{d}{n+1} v_K(\det\E),
		\]
		for $w\in\RR^{n+1}$.
		\item
		The restriction of $\phi_X$ to each apartment $\A_\E$ is uniformly continuous, piecewise affine and convex.
		\item
		$\phi_X$ is uniformly continuous.
	\end{enumerate}
\end{lemma}

\begin{proof}
	Set $v:=v_{\E,w}$. Then by \eqref{eq:phi_X_def}, Lemma \ref{lem:omega} (i), \eqref{eq:v_tilde} and \eqref{eq:linear_form_l_i} we have
	\[\begin{split}
		\phi_X([v]) &= d\cdot\omega(v) - v(F) \\
		&= \frac{d}{n+1}\big(w_0+\ldots+w_n - v_K(\det\E)\big) - \min_{i\in I_d} (v_K(a_i) - \gen{i,w}) \\
		&= \max_{i\in I_d} \big( l_i(w) - v_K(a_i)\big) -\frac{d}{n+1} v_K(\det\E),
	\end{split}\]
	proving (i). In particular, the restriction of $\phi_X$ to each apartment is the maximum of a finite number of affine functions, from which (ii) follows immediately. Moreover, the linear parts of the affine functions are taken from the finite set of linear forms $l_i$, for $i\in I_d$, which is independent of $\E$. This shows that for $w,w'\in\EE_n$ we have
	\[
	\frac{\abs{\phi_{F,\E}(w) - \phi_{F,\E}(w')}}{\norm{w-w'}} \leq C,
	\]
	for a constant $C$ independent from $\E$.
	Since any two points on $\B_K$ lie on one apartment (Proposition \ref{prop:v_diagonalize}), this proves that $\phi_X$ is uniformly continuous. This completes the proof of Lemma \ref{lem:stability_function1}.
\end{proof}

The next lemma proves the remaining claim from Proposition \ref{prop:stability_function_intro} (i), and prepares the proof of Part (ii).

\begin{lemma} \label{lem:stability_function2}
	Assume that $X=V_+(F)$ is semistable. Then the following holds.
	\begin{enumerate}[(i)]
		\item
		There exists a constant $C$ such that for all $g\in\GL(V)$ we have
		\[
		v_0(\lexp{g}{F}) -\frac{d}{n+1}v_K(\det g) \geq C.
		\]
		\item
		The function $\phi_X$ is bounded from below.
		\item
		If $X=V_+(F)$ is stable, then the restriction of $\phi_{F,\E}$ to any apartment is radially unbounded.
	\end{enumerate}
\end{lemma}

\begin{proof}
	For this proof we may assume that $v_K$ is normalized so that $\Gamma_K=\ZZ$.
	
	Let $I\in K[W]_r\backslash\{0\}$ be an $\SL(V)$-invariant of degree $r$, for the representation $W:=K[V]_d$. Then for any $g\in\GL(V)$ and $c\in K$ we have
	\begin{equation} \label{eq:transformation_invariant}
		I(c\cdot\lexp{g}{F}) = c^r\cdot \det(g)^{\frac{rd}{n+1}} I(F),
	\end{equation}
	see e.g.\ \cite[Lemma 4.5]{Kollar97}. Since we assume $X=V_+(F)$ to be semistable, there exists such an invariant $I$ with $I(F)\neq 0$, see \cite[Chapter 2]{MumfordGIT}. After multiplying $I$ with a suitable constant, we may also assume that for all $\tilde{F}\in W$ we have
	\begin{equation} \label{eq:normalize_invariant}
		v_0(\tilde{F})\geq 0 \quad\Rightarrow\quad v_K(I(\tilde{F})) \geq 0.
	\end{equation}
	To see this, note that the monomials of degree $d$ in the elements of our standard basis $\E_0$ form a basis of the $K$-vectors space $W=S^d(V^*)$. If we write $I$ as a polynomial in the dual of that basis, and take care that the coefficients of $I$ are in $\OO_K$, then \eqref{eq:normalize_invariant} holds.
	
	Let $g\in GL(V)$ be arbitrary. We set
	\[
	\tilde{F}:=\pi^{-m}\cdot\lexp{g}{F}, \quad m:=v_0(\lexp{g}{F}).
	\]
	Then $v_0(\tilde{F})=0$. Using \eqref{eq:normalize_invariant} and \eqref{eq:transformation_invariant} we obtain the inequality
	\begin{equation}
		0 \leq v_K(I(\tilde{F})) = -r\cdot v_0(\lexp{g}{F}) +\frac{rd}{n+1}v_K(\det g) + v_K(I(F)),
	\end{equation}
	from which we conclude
	\begin{equation} \label{eq:lower_bound}
		v_0(\lexp{g}{F}) - \frac{d}{n+1}v_K(\det g)  \geq -\frac{1}{r}v_K(I(F))=:C,
	\end{equation}
	for all $g\in GL(V)$, proving (i).
	
	Since $\phi_X$ is uniformly continuous by Lemma \ref{lem:stability_function1} (iii), it suffices for the proof of (ii) to show that $\phi_X$ is bounded from below on the set of vertices of $\B_K$. Any $v\in\N_K^\circ$ is of the form $v=\lexp{g}{v}_0$, for some $g\in\GL(V)$. Using Lemma \ref{lem:omega} (iii), \eqref{eq:v_tilde_g} and \eqref{eq:lower_bound} we obtain the inequality
	\[\begin{split}
		\phi_X([v]) &= \omega(\lexp{g}{v}_0) - (\lexp{g}{v}_0)(F) \\
		& = \omega(v_0) -\frac{d}{n+1}v_K(\det g) - v_0(\lexp{g^{-1}}{F}) \\
		&\geq C,
	\end{split}\]
	proving (ii).
	
	For the proof of (iii) we fix a basis $\E$. Recall the definition of the subset $I_{F,\E}\subset I_d$ from Lemma \ref{lem:stability_function1} (i), which is also used in \S \ref{subsec:stability_definitions}. Let $w=(w_0,\ldots,w_n)\in\RR^{n+1}\neq 0$ be balanced, i.e.\ $w_0+\ldots+w_n=0$. By Theorem \ref{thm:numerical_criterion} and \eqref{eq:stability2} there exists an element $i\in I_{F,\E}$ such that $l_i(w) = \gen{i,w}<0$. Using the formula for $\phi_{F,\E}$ from Lemma \ref{lem:stability_function1} (i) we see that
	\[
	\phi_{F,\E}(t\cdot w) \longrightarrow \infty \;\; \text{for $t\to-\infty$.}
	\]
	This shows that $\phi_{F,\E}$ is radially unbounded and concludes the proof of Lemma \ref{lem:stability_function2}.
\end{proof}

The next lemma proves Part (ii) of Proposition \ref{prop:stability_function_intro}.

\begin{lemma} \label{lem:stability_function3}
	We assume that $X=V_+(F)$ is stable. Then the function $\phi_X$ achieves a global minimum in a rational point of $\B_K$.
\end{lemma}

\begin{proof}
	We first consider the restriction $\phi_{F,\E}$ of $\phi_X$ to the apartment $\A_\E$, for an arbitrary basis $\E$. By Lemma \ref{lem:stability_function1}, $\phi_{F,\E}$ is continuous and convex, and by Lemma \ref{lem:stability_function2} (iii) it is radially unbounded. The existence of a minimum $m_\E$ of the function $\phi_{F,\E}$ is therefore clear. Moreover, by the explicit formula for $\phi_{F,\E}$ from Lemma \ref{lem:stability_function1} (i) we see that the set of points $w\in\RR^{n+1}$ where this minimum is achieved is the solution set of a linear optimization problem of the form
	\[
	A\cdot(w,t)^T \leq b_\E, \quad \text{minimize $t$,}
	\]
	where $A$ is an $(\abs{I_d}, n+2)$-matrix with integral coefficients, independent of $\E$, and $b_\E$ is a vector with entries in a discrete subgroup $N_K\subset\QQ$, which is also independent of $\E$ (but depends on the group $\Gamma_K$). It is well known that if such an optimization problem has a solution, then it has an optimal solution $(w,m_\E)$ with entries in a discrete subgroup $N:=\frac{1}{c}N_K\subset\QQ$, and where $c\in\NN$ only depends on the matrix $A$.\footnote{The idea of the proof is that an optimal solution is always achieved in a vertex of the polytope defined by the inequality $A\cdot(w,t)^T\leq b_\E$, see \cite[\S 3.2]{ziegler_lectures}. Such a vertex corresponds to a solution of a linear system of equations given by a subset of the rows of the system of equations $A\cdot(w,t)^T=b_\E$.}
	
	So the set of minima $m_\E$, where $\E$ runs over all bases, is contained in the discrete group $N$, which is independent of $\E$. But it is also bounded from below, by Lemma \ref{lem:stability_function2} (ii). It follows that $\phi_X$ achieves a global minimum $m_X=\min_\E m_\E$. As each $m_\E$ is attained in a rational point, the lemma is proved.
\end{proof}

Finally, we prove the remaining Part (iii) of Proposition \ref{prop:stability_function_intro}.

\begin{lemma} \label{lem:stability_function4}
	Let
	\[
	M_X:= \{ b\in\B_K \mid \phi_X(b)=m_X\}.
	\]
	be the set where $\phi_X$ attains its global minimum $m_X$.
	Let $b\in\B_K^0$ be a vertex, corresponding to the hypersurface model $\X$. Then $\X$ is semistable if and only if $b\in M_X$. Moreover, $\X$ is stable if and only if $M_X=\{b\}$ over any finite extension $L/K$.
\end{lemma}

\begin{proof}
	Let $b\in\B_K^\circ$ be a vertex of $\B_K$. Since any two points of $\B_K$ lie in a common apartment (Proposition \ref{prop:v_diagonalize}), $\phi_X$ attains its global minimum $m_X$ in $b$ if and only if the restrictions $\phi_{F,\E}$ attain their minimum $m_\E$ in $b$, for every apartment $\A_\E$ containing $b$.
	
	Fix one apartment $\A_\E$ containing $b$. We may assume that $b=[v]$, with $v=v_{\E,0}$. Then $\phi_X(b)=\phi_{F,\E}(0)$. As $\phi_{F,\E}$ is convex by Lemma \ref{lem:stability_function1} (ii), $\phi_{F,\E}(0)$ is minimal if and only if for all $w=(w_0,\ldots,w_n)\neq 0$ the function
	\[
	f_{\E,w}:\RR_{\geq 0}\to\RR, \quad f_{\E,w}(t):=\phi_{F,\E}(t\cdot w)
	\]
	is nondecreasing on some interval $[0,\epsilon)$, with $\epsilon>0$ (and then it is nondecreasing on $\RR_{\geq 0}$). Here we may assume that $w_0+\ldots+w_n=0$.
	
	Write $\E=(x_0,\ldots,x_n)$ and $F=\sum_{i\in I_d} a_ix_0^{i_0}\cdots x_n^{i_n}$.
	By Lemma \ref{lem:stability_function1} (i) we have
	\begin{equation} \label{eq:f(t)}
		f_{\E,w}(t) = \max_{i\in I_{F,\E}} \big( \gen{i,w} - v_K(a_i) \big),
	\end{equation}
	where the subset $I_{F,\E}\subset I_d$ contains the $i\in I_d$ such that $v_K(a_i)\neq 0$. In particular,
	\begin{equation} \label{eq:f(0)}
		f_{\E,w}(0) = -\min_{i\in I_{F,\E}} v_K(a_i) = - v(F).
	\end{equation}
	Consider the subset
	\[
	\bar{I}_{F,\E} := \{ i\in I_{F,\E} \mid v_K(a_i)=v(F)\}.
	\]
	Then \eqref{eq:f(t)} shows that $f_{\E,w}$ is nondecreasing on some interval $[0,\epsilon)$ if and only if
	\begin{equation} \label{eq:Fbar_is_semistable}
		\gen{i,w} \geq  0,\quad \text{for all $i\in \bar{I}_{F,\E}$.}
	\end{equation}
	On the other hand, the hypersurface model $\X$ of $X$ corresponding to $b$ is given by the equation $F_b:=\pi^{-v(F)}F\in \OO_K[x_0,\ldots,x_n]$, and its special fiber $\X_s$ is therefore given by the equation $\overline{F}_b\in k[x_0,\ldots,x_n]$, the reduction of $F_b$. By the definition of the subset $\bar{I}_{F,\E}$ we have
	\[
	\overline{F}_b = \sum_{i\in \bar{I}_{F,\E}} \bar{a}_i x_0^{i_0}\cdots x_n^{i_n}, \quad \text{with $\bar{a}_i\neq 0$.}
	\]
	Now Corollary \ref{cor:instability} shows that $\X_s=V_+(\overline{F}_b)$ is semistable if and only if \eqref{eq:Fbar_is_semistable} holds for all bases $\E$ such that $v=v_{\E,0}$ and for all $w=(w_0,\ldots,w_n)\neq 0$ with $w_0+\ldots+w_n=0$. As we have already shown, the latter holds if and only if $\phi_X$ attains its global minimum in $b$.
	
	It remains to prove that $\X_s$ is stable if and only if $b$ is an isolated minimum that persists under any finite extension $L/K$. We can use basically the same argument as before, with the following modifications. The hypersurface $\X_s=V_+(\overline{F}_b)$ is stable if and only if \eqref{eq:Fbar_is_semistable} holds with strict inequality, for all $w$ and all bases $\E$ of $V^* \otimes_K L$ for all finite extensions $L/K$. This condition is satisfied if and only if the function $f_{\E,w}$ is strictly increasing on some interval $[0,\epsilon)$ for all $\E$ and $L$. Clearly, this is equivalent to $b$ being an isolated minimum of the function $\phi_X$ considered on $\B_L$ for any finite extension $L/K$.
	
	This concludes the proof of Lemma \ref{lem:stability_function4} and of Proposition \ref{prop:stability_function_intro}.
\end{proof}

\subsection{Finding the minimum}\label{subsec:finding_minimum}

Assume that we have a solution of Problem~\ref{prob:find_instability} at our disposal.
Then the proof of Proposition~\ref{prop:stability_function_intro} suggests an explicit
procedure for finding the minimal value $m_X$ of the stability function $\phi_X$ on
$\B_K$ and a rational point $b\in\B_K(\QQ)$ where this minimum is attained. Here is a brief sketch of this algorithm. Full details will be given in \cite{KletusDiss}, and an implementation is available in \cite{KletusGitHub}.

One starts with an arbitrary apartment
$\A\subset\B_K$ and determines the subset $M_X\cap \A$ by linear optimization, using
that $\phi_X|_{\A}$ is an explicit piecewise affine and convex function
(Lemma \ref{lem:stability_function1} (i)).
One then chooses a rational point $b\in M_X\cap\A$ and (formally) passes to a finite
ramified extension $L/K$ such that $\iota(b)\in\B_L$ becomes a vertex, in order to compute the special fiber $\X_{b,s}$ of the hypersurface model $\X_b$.\footnote{It is not necessary to actually choose and work with the extension $L/K$. One can compute the special fiber $\X_{b,s}$ -- which is independent of the choice of $L/K$ -- using the concept of {\em graded reduction} (\cite{KletusDiss}).}

Using a solution to Problem~\ref{prob:find_instability}, one can now decide whether
$\X_b$ is semistable. If this is the case, the algorithm terminates, since $b$ already
lies in $M_X$. Otherwise, an explicit instability of $\X_{b,s}$ is produced.
Such an instability corresponds canonically to a point of the spherical building, which may be
identified with a direction in the tangent cone of $\B_K$ at $b$ along which
$\phi_X$ strictly decreases (see the proof of Lemma \ref{lem:stability_function4}). This direction determines a new apartment
$\A'\subset\B_K$ containing $b$ on which the minimum of $\phi_X$ is strictly smaller
than on $\A$.

One then replaces $\A$ by $\A'$ and repeats the procedure.
The algorithm must terminate after finitely many steps, since $\phi_X$ is bounded
from below and the successive improvements of the minimum occur in a discrete set
(Lemma \ref{lem:stability_function3}).

\begin{remark}
	As described above, the algorithm does not determine the full set $M_X\subset \B_K$ where $\phi_X$ achieves its minimum. If the residue field $k$ is finite -- and hence $K$ is locally compact -- an easy improvement of the above algorithm can do this. But in general, this is a delicate issue, as the set $M_X$ may not be contained in any finite union of apartments. See \S \ref{subsec:example3} for an explicit example illustrating this phenomenon.	
\end{remark}

\subsection{The tame case} \label{subsec:tame_case} 

In general, the algorithm from \S \ref{subsec:finding_minimum} gives only a partial solution to Problem \ref{prob:semistable_model}. More precisely, if $X$ has a semistable model over $\OO_K$ then the algorithm finds it -- it corresponds to a vertex of $\B_K$ lying inside the set $M_X$. If the set $M_X$ contains no vertex, then no semistable model exists over $K$, and we have to pass to a suitable  finite extension $L/K$. If $L/K$ is sufficiently ramified, then $M_X$ contains a point which becomes a vertex in $\B_L$. But this vertex may not be a minimum of the stability function on the larger building $\B_L$, and we have to continue with the algorithm over the field $L$. It is not obvious how to choose the extension $L/K$ such that the process of choosing successive extensions terminates. 

We are going to show that this is not a problem if $X$ has a semistable model over an extension $L/K$ which is at most tamely ramified (see Theorem \ref{thm:tame_case_intro} from the introduction). In particular, if the residue characteristic of $K$ is zero, then our first algorithm completely solves Problem \ref{prob:semistable_model} (assuming we have a solution for Problem \ref{prob:find_instability}). The general case is much more involved and discussed further in \S \ref{sec:stable_minimum} below.

We recall the terminology introduced in the introduction.
A point $b\in\B_K$ at which the stability function $\phi_X$ attains its minimum
$m_{X,K}$ is called a \emph{stable minimizer} if, for every finite extension $L/K$,
its image $\iota(b)\in\B_L$ is still a minimizer of $\phi_X$. Then $m_{X,K}$ is called the {\em stable minimum} of $\phi_X$. 
By Proposition~\ref{prop:stable_minimum} (i), the existence of a semistable hypersurface
model over a finite extension $L/K$ implies that $m_{X,L}$ is the
stable minimum of $\phi_X$. Part (ii) of the proposition states that a rational stable minimizer which becomes a vertex after a finite extension $L/K$ corresponds to a semistable model of $X$ over $L$. 

We briefly indicate the proof of Proposition~\ref{prop:stable_minimum},
since it will be used repeatedly in the sequel.

\begin{proof}
	(i) Let $X_L$ have a semistable hypersurface model over $\OO_L$, corresponding to a vertex
	$v\in\B_L^\circ$. Then $\phi_X(v)=m_{X,L}$ by Proposition~\ref{prop:stability_function_intro}(iii).
	For any finite extension $L'/L$, the base change of a semistable model is again semistable
	(GIT semistability of the special fiber is preserved under base change), hence the image of
	$v$ in $\B_{L'}$ is again a vertex in the minimum locus of $\phi_X$ on $\B_{L'}$.
	Therefore $m_{X,L'}\le \phi_X(v)=m_{X,L}$. Since always $m_{X,L'}\ge m_{X,L}$ by functoriality of
	$\phi_X$, we get $m_{X,L'}=m_{X,L}$ for all $L'/L$.
	
	(ii) By definition, $\iota(b)$ is a minimizer of $\phi_X$ on $\B_L$. Since it is a vertex, the
	corresponding hypersurface model is semistable by Proposition~\ref{prop:stability_function_intro}(iii).
\end{proof}

\begin{lemma} \label{lem:tame_case}
	Let $L/K$ be a finite Galois extension, with Galois group $\Gamma:=\Gal(L/K)$. Then
	\[
	M_L\cap\B_L^\Gamma \neq \emptyset,
	\]
	i.e.\ $M_L$ contains a point fixed by $\Gamma$.
\end{lemma}

\begin{proof}
	The action of $\Gamma$ on $\B_L$ preserves $\phi_X$, hence preserves $M_L$.
	Moreover, $M_L$ is closed and convex because every geodesic segment in
	$\B_L$ is contained in an apartment and $\phi_X$ is convex on apartments.
	The lemma follows now from the Bruhat-Tits Fixed-Point Theorem (\cite[\S VI.4]{brown2008buildings}).
\end{proof}

\begin{theorem} \label{thm:tame_case}
	Assume that $X$ has a semistable hypersurface model
	over a tamely ramified finite extension.
	Then every minimizer $b\in \B_K$ of $\phi_X$
	is a stable minimizer.
\end{theorem}

\begin{proof}
	Let $L/K$ be a finite and tamely ramified extension over which $X$ has a semistable model. We may assume that $L/K$ is Galois and set $\Gamma:=\Gal(L/K)$. By Proposition \ref{prop:stable_minimum} (i), $m_{X,L}$ is the stable minimum for $\phi_X$.
	
	By Lemma \ref{lem:tame_case}, $M_L\cap\B_L^\Gamma$ is nonempty. Moreover, by the theorem of Rousseau--Prasad~\cite{prasad01}, the natural embedding $\iota:\B_K\hookrightarrow\B_L$ induces an isomorphism
	\[
	\B_K\liso \B_L^\Gamma.
	\]
	It follows that there exists $b\in\B_K$ such that $\phi_X(b)=m_{X,L}$.
	We conclude that $m_{X,K}=m_{X,L}$ is the stable minimum of $\phi_X$, which proves the theorem.
\end{proof}

Combining Theorem \ref{thm:tame_case} with Proposition \ref{prop:stable_minimum} (ii) and Theorem \ref{thm:semistable_reduction}, we immediately obtain the following result, corresponding to Theorem \ref{thm:tame_case_intro} from the introduction.

\begin{corollary} \label{cor:tame_case}
	Assume that the residue field $k$ has characteristic $0$.
	Let $b\in \B_K(\QQ)$ be a rational point where $\phi_X$ attains its minimum.
	Let $L/K$ be a finite extension such that the image of $b$ in $\B_L$ is a vertex.
	Then $b$ corresponds to a semistable hypersurface model of $X_L$ over $\OO_L$.	
\end{corollary}

\section{Existence of a stable minimum} \label{sec:stable_minimum}

The goal of this section is to prove Theorem~\ref{thm:global_minimum}, which asserts
the existence of a global minimum 
\[
\min_{L/K} m_{X,L}
\]
of the stability function $\phi_X$ among all finite
extensions $L/K$. As a direct consequence, we obtain a new proof of the semistable
reduction theorem (Theorem~\ref{thm:semistable_reduction}).

Our proof does not use Geometric Invariant Theory. Instead, it relies on two elementary
ingredients: a variant of the maximum principle from non-archimedean analysis
(Lemma~\ref{lem:maximum_principle}), and a combinatorial lemma from polyhedral
geometry (Lemma~\ref{lem:l_i_inequalities}). These results form the theoretical basis
for the second stage of our approach, which ensures the existence of a field extension
over which the stability function attains a stable minimum. When combined with the
algorithm from \S~\ref{subsec:finding_minimum}, this leads to a complete and practical
algorithm for solving Problem~\ref{prob:semistable_model} for plane curves; see
\cite{KletusDiss} for details and \cite{KletusGitHub} for an implementation.

\subsection{A variant of the maximum principle} \label{subsec:affinoids}

Let $X=\Spec A$ be an affine variety over $K$.
We let $X^\rig$ denote the rigid-analytic $K$-space attached to $X$, as in \cite[\S 5.4]{BoschRigid}. The underlying set of $X^\rig$ is simply the set of closed points of the scheme $X$.

Given a point $x\in X^\rig$, the residue field $L:=K(x)$ is a finite extension of $K$, which is naturally equipped with the
unique discrete valuation $v_L$ extending $v_K$. Therefore, the point $x\in X^\rig$ induces a map
\[
v_x:A\to\RR\cup\{\infty\}, \quad f \mapsto v_L(f(x)).
\]

\begin{lemma} \label{lem:maximum_principle}
	Let $f_1,\ldots,f_N\in A$, and let $U\subset X^\rig$ be an affinoid subdomain. Then the restriction of the function
	\[
	\alpha:X^\rig\to\RR\cup\{\infty\}, \quad
	\alpha(x) := \min\big(v_x(f_1),\ldots,v_x(f_N)\big),
	\]
	to $U$ assumes its maximum. This maximum is finite if and only if the $f_i$ do not have a common zero in $U$.
\end{lemma}

\begin{proof}
	This is \cite[\S 7.3.4, Lemma 7]{BGR} (as we work with additive valuations instead of absolute values, the roles of $\min$ and $\max$ relative to {\em loc.cit.} are inverted.)
\end{proof}

\subsection{Some inequalities} \label{subsec:inequalities}

We fix $n\geq 1$ and $d\geq 3$. Recall the definition of the linear forms
\[
l_i(w):= \frac{d}{n+1}(w_0+\ldots+w_n) - (i_0w_0+\ldots +i_nw_n)
\]
on $\RR^{n+1}$, for $i=(i_0,\ldots,i_n)\in I_d$, introduced in \S \ref{subsec:proof_of_prop_1.5}.

\begin{lemma} \label{lem:l_i_inequalities}
	There exist nonnegative integers $ e_{i,k}\in\NN_0$, for $i\in I_d$ and $k=1,\ldots,N$, and a positive integer $e>0$ such that the following holds:
	\begin{enumerate}[(i)]
		\item
		For a tuple $(t_i)\in\RR^{I_d}$,
		\begin{equation} \label{eq:l_i_inequalities}
			t_i\geq l_i(w)
		\end{equation}
		holds for all $i\in I_d$ and some $w\in\RR^{n+1}$, independent of $i$, if and only if
		\[
		\sum_{i\in I_d}  e_{i,k} t_i \geq 0, \quad k=1,\ldots,N.
		\]
		\item
		For all $k$,
		\[
		e= \sum_{i\in I_d} e_{i,k}.
		\]
	\end{enumerate}
\end{lemma}

\begin{proof}
	Let $P\subset \RR^{I_d}\times\RR^{n+1}$ be the set of tuples $(t,w)$ such that \eqref{eq:l_i_inequalities} holds for all $i\in I_d$. Let
	$Q:=p(P)\subset\RR^{I_d}$ be the image of $P$ under the projection $p:\RR^{I_d}\times\RR^{n+1}\to\RR^{I_d}$. Then $(t_i)\in Q$ if and only if \eqref{eq:l_i_inequalities} holds for all $i\in I_d$ and some $w\in\RR^{n+1}$, independent of $i$.
	
	As $P$ is defined by finitely many rational linear inequalities, it is a rational polyhedron (\cite[Lecture 1]{ziegler_lectures}). It is shown in \S 1.2 of \emph{loc.cit.} that $Q$ is again a rational polyhedron and therefore defined by linear inequalities
	\[
	\sum_{i\in I_d}  e_{i,k}t_i \geq 0,\quad k=1,\ldots,N,
	\]
	with $ e_{i,k}\in\ZZ$. Without loss of generality, we may assume that for all $k$ there exists at least one $i\in\ I_d$ such that $e_{i,k}\neq 0$. Looking at the inequalities \eqref{eq:l_i_inequalities} defining $Q$ we see that for $(t_i)\in Q$ and $(t_i')\in\RR^{I_d}$ with $t_i\leq t_i'$ for all $i$ we also have $(t_i')\in Q$. This implies that $ e_{i,k}\geq 0$, for all $i,k$.
	
	Set
	\[
	e_k := \sum_{i\in\I_d} e_{i,k}, \quad k=1,\ldots,N
	\]
	and $e:={\rm lcm}_k(e_k)$. Clearly $e_k,e>0$. After multiplying $e_{i,k}$ by $e/e_k$ (which does not affect the validity of (i)), (ii) holds as well.
\end{proof}

\subsection{Proof of Theorem \ref{thm:global_minimum}} \label{subsec:proof_of_global_minimum}

In this section, we regard the group $G:=\SL_{n+1}$ as an algebraic group over $K$. Thus,
\[
G = \Spec A, \quad A = K[g_{\mu,\nu} \mid 0\leq \mu,\nu\leq n]/(\det(g_{\mu,\nu})-1).
\]

We fix a homogeneous form $F\in K[x_0,\ldots,x_n]$ of degree $d$ which defines a stable hypersurface $X=V_+(F)\subset\PP^n_K$.  Let $L/K$ be an arbitrary field extension and $g=(g_{\mu,\nu})\in G(L)$. Write
\[
\lexp{g}{F} := F\circ g^{-1} = \sum_{i\in I_d} \tilde{a}_i(g)x_0^{i_0}\cdots x_n^{i_n} \in L[x_0,\ldots,x_n]_d.
\]
Then the functions
\[
\tilde{a}_i:G(L)\to L
\]
are polynomials over $K$ in the coefficients $g_{\mu,\nu}$ of $g$. In other words, we may regard $\tilde{a}_i$ as elements of the Hopf algebra $A$ of $G$.

Set
\begin{equation} \label{eq:f_k}
	f_k := \prod_{i\in I_d} \tilde{a}_i^{ e_{i,k}} \in A, \quad k=1,\ldots,N,
\end{equation}
where $e_{i,k}\in\NN_0$ are the nonnegative integers from Lemma \ref{lem:l_i_inequalities}. By Part (ii) of this lemma,
\begin{equation} \label{eq:e_k}
	e:=\sum_{i\in I_d} e_{i,k}
\end{equation}
is positive and independent from $k$. We consider the function
\[
\alpha:G^\rig\to\RR\cup\{\infty\}, \quad
\alpha(g) :=\min\big(v_g(f_1),\ldots,v_g(f_N)\big),
\]
as in Lemma \ref{lem:maximum_principle}.

We fix a basis $\E_0$ of $V^*$ and let $\A_0:=\A_{\E_0}$ denote the corresponding apartment of $\B_K$. Let $L/K$ be a finite extension and $g\in G(L)\subset G^\rig$. Then $\E_g:=g(\E_0)$ is the basis of $V_L^*$ defining the apartment $\A_g:=g(\A_0)$ of $\B_L$. Let $m_g$ denote the minimum that the stability function $\phi_X$ attains on $\A_g$ (by Lemma \ref{lem:stability_function2}).

\begin{lemma} \label{lem:minimum}
	For $g\in G(L)$ as above we have
	\[
	m_g = -\alpha(g)/e.
	\]
\end{lemma}

\begin{proof}
	In order to prove the lemma, it suffices to show that
	\begin{equation} \label{eq:minimum1}
		m_g \leq -\frac{s}{e} \quad \Leftrightarrow\quad
		\alpha(g)\geq s.
	\end{equation}
	By Lemma \ref{lem:stability_function1} (i), the restriction of $\phi_X$ to $\A_g\cong\RR^{n+1}/\RR\cdot(1,\ldots,1)$ is given by the formula
	\begin{equation} \label{eq:minimum2}
		\phi_X(w) = \max_{i\in I_d} \big( l_i(w) - v_g(\tilde{a}_i) \big),
	\end{equation}
	where $w\in \RR^{n+1}$ and the linear form $l_i$ is defined by \eqref{eq:linear_form_l_i}. Clearly, \eqref{eq:minimum2} shows that $m_g\leq -s/e$  if and only if there exists some $w\in\RR^{n+1}$ such that the inequalities
	\begin{equation} \label{eq:minimum3}
		t_i:=v_g(\tilde{a}_i)-s/e\geq l_i(w)
	\end{equation}
	hold. In other words, if the tuple $(t_i)$ is contained in the set $Q\subset\RR^{I_d}$ defined in \S \ref{subsec:inequalities}, see \eqref{eq:l_i_inequalities}. Now Lemma \ref{lem:l_i_inequalities} shows that $m_g\leq -s/e$ if and only
	\begin{equation} \label{eq:minimum4}
		\sum_{i\in I_d} e_{i,k} t_i \geq 0, \quad k=1,\ldots,N.
	\end{equation}
	By the definition of the functions $f_k$ (see \eqref{eq:f_k}) and the integer $e$ (see \eqref{eq:e_k}), we have
	\begin{equation} \label{eq:minimum5}
		v_g(f_k) = \sum_{i\in I_d} e_{i,k}v_g(\tilde{a}_i)
		= \sum_{i\in I_d} e_{i,k}t_i  + s.
	\end{equation}
	On the other hand, $\alpha(g)\geq s$ if and only if
	\begin{equation} \label{eq:minimum6}
		v_g(f_k) \geq s, \quad k=1,\ldots,N.
	\end{equation}
	The equivalence \eqref{eq:minimum1} follows by combining \eqref{eq:minimum4}, \eqref{eq:minimum5} and \eqref{eq:minimum6}.
\end{proof}

We claim that the function $\alpha:G^\rig\to\RR$ attains its maximum in some $g^{\min}\in G^\rig$. By Lemma \ref{lem:minimum} this means that the minimum
\[
m := \min_{g} m_g = m_{g^{\min}}
\]
is attained in $g^{\min}$. This will prove Theorem \ref{thm:global_minimum}, since the minimal value of $\phi_X$ on $\B_L$ is 
\[
m_{X,L} = \min\{m_g \mid g\in G(L)\}.
\]

To prove the claim, we would like to apply Lemma \ref{lem:maximum_principle}. But in order to do this, we first have to show that we may restrict $\alpha$  to a certain affinoid domain of $G^\rig$.

We set
\[
U := \{ x\in G^\rig \mid v_x(g_{\mu,\nu})\geq 0 \}.
\]
As the $g_{\mu,\nu}$ form a generating system of the $K$-algebra $A$, $U$ is an affinoid subdomain, of the form
\[
U = {\rm Sp}(\A), \quad \A := K\gen{g_{\mu,\nu} \mid 0\leq \mu,\nu\leq n}/(\det(g_{\mu,\nu})-1).
\]
See \cite[\S 5.4]{BoschRigid}.

Let $L/K$ be a fixed finite extension. Then by definition,
\[
U(L) = \SL_{n+1}(\OO_L) \subset G(L).
\]
This group is the stabilizer of the vertex $[v_0]\in\B_L^\circ$ corresponding to the homothety class of $v_0:=v_{\E_0,0}$ (or, equivalently, the homothety class of the lattice in $V_L^*$ generated by the basis $\E_0$), see \cite[Cor. 3.4]{parreau_immeubles}.

Let $g\in G(L)$ and consider the apartment $\A_g:=g(\A_0)\subset\B_L$. Let $[v]\in\A_g$ be a point in which the minimum $m_g$ of $\phi_X$ in $\A_g$ is attained. Let $\A'\subset \B_K$ be an apartment which contains both $[v_0]$ and $[v]$. We claim that there exists a $g'\in U(L)$ such that $\A'=g'(\A_0)$.

To prove the claim, we write $\A'=\A_{\E'}$, for some basis $\E'$ of $V_L^*$. Then $v_0=v_{\E',w'}$, with $w'\in\Gamma_L^{n+1}$ because $[v_0]$ is a vertex of $\B_L$. After multiplying the elements of $\E'$ with suitable constants, we may assume that $w'=0$. But now
\[
v_0 = v_{\E_0,0} = v_{\E',0},
\]
which means that $\E_0$ and $\E'$ generate the same $\OO_L$-lattice of $V_L^*$. This shows that $\E'=g'(\E_0)$ for some $g'\in U(L)$, proving the claim.

Since $[v]\in\A'=g'(\A_0)$, the choice of $[v]$ shows that
\[
m_{g'} \leq m_g
\]
and therefore
\[
\alpha(g') \geq \alpha(g),
\]
by Lemma \ref{lem:minimum}. This shows that if $\alpha$ attains its maximum on $G^\rig$, it attains it on $U$.

We can now apply Lemma \ref{lem:maximum_principle} to prove that $\alpha$ does indeed attain its global maximum (on the affinoid $U$). As explained before, this concludes the proof of Theorem \ref{thm:semistable_reduction}.

\section{Examples} \label{sec:examples}

We study three explicit examples of smooth plane curves over a complete and discretely valued non-archimedean field $K$. These examples are chosen to illustrate a few salient points of our approach, while being relatively simple. We also use this opportunity to demonstrate the implementation from \cite{KletusGitHub}. The common theme is that we try to understand the exact relationship between two related problems: on the one hand the computation of a semistable plane model of a plane curve (Problem \ref{prob:semistable_model} for $n=2$) and on the other hand the computation of a geometric semistable model (see e.g.\ \cite{oberwolfach} for a recent survey of the state of the art, and \cite{SSW} for a solution in case of plane quartics).

In all three examples we assume that the residue field $k$ of $K$ is algebraically closed. If $X$ is a smooth projective curve over $K$ of genus $g\geq 2$, then there exists a minimal finite extension $L/K$ over which $X$ has geometric semistable reduction. As we will see with our second example, the analogous statement for plane curves and GIT-semistable plane model does not hold.

\subsection{A cubic with potentially good reduction} \label{subsec:example1}

Let $K:=\QQ_2^\nr$ be the maximal unramified extension of the field of $2$-adic numbers, and
\begin{equation} \label{eq:exa1.1}
	F := x_0^2x_2-x_1^3+2x_2^3 \in K[x_0,x_1,x_2].
\end{equation}
Then $X:=V_+(F)\subset\PP^2_K$ is a smooth cubic. In fact, using the dehomogenization $x_2:=1$ and setting $x:=x_1$, $y:=x_0$, we obtain the familiar Weierstrass equation $y^2=x^3-2$ of an elliptic curve. Using standard methods one shows that $E$ has potentially good reduction over a unique quadratic extension $L/K$. In particular, the cubic $X$ has a unique stable model over $L$ but no semistable model over $K$. In the following we will prove this using the methods developed in this article. For this we follow the algorithm discussed in \S \ref{subsec:finding_minimum}. Where necessary, we pass to an unspecified extension $L/K$ with certain properties. In this simple example we manage to `guess' the right extension $L/K$, and to show that it is unique.

The naive $\OO_K$-model $\X_0$ of $X$, corresponding to the choice of the coordinate system $\E_0=(x_0,x_1,x_2)$, has special fiber $\X_{0,s}=V_+(\bar{F})\subset \PP^2_{\FF_2}$, with
\[
\bar{F} = x_0^2x_2+x_1^3.
\]
It is a singular cubic, with a cusp at the point $P:=[0:0:1]$. More precisely, $P$ is a point on $\X_{0,s}$ of multiplicity $2$, and its tangent cone is the double line given by $x_0^2=0$. One checks that $((x_0,x_1,x_2),(2,1,-3))$ is an instability of $\X_{0,s}$ (Case (c) of Proposition \ref{prop:plane_curves}). So the model $\X_0$ is \emph{not} semistable.

We note in passing that $\X_0$ is a {\em minimal model} of $X$. This can be checked with the algorithm from \cite[\S 7]{ES}, or the usual method for checking whether a Weierstrass equation is minimal. It does not immediately follow from the reasoning given below.

By Lemma \ref{lem:stability_function1} (i) and inspection of the Equation \eqref{eq:exa1.1}, the restriction $\phi_{F,\E_0}$ of the stability function to the apartment $\A_{\E_0}$ is
\begin{equation} \label{eq:exa1.2}
	\phi_{F,\E_0}([w]) = \max\big( -w_0+w_1, w_0-2w_1+w_2, w_0+w_1-2w_2-1 \big).
\end{equation}
In the following we use the identification $\RR^2\cong\EE_2$, $(w_0,w_1)\mapsto [(w_0,w_1,0)]$, which simplifies \eqref{eq:exa1.2} to
\begin{equation} \label{exa1.3}
	\phi_{F,\E_0}(w_0,w_1) = \max\big( -w_0+w_1, w_0-2w_1, w_0+w_1-1 \big).
\end{equation}
It is easy to see that the function $\phi_{F,\E_0}$ attains its minimal value $m_{\E_0}=-1/6$ in the point $b_1=(1/2,1/3)$, and nowhere else.
As the point $b_1$ lies inside the $2$-simplex with vertices $(0,1),(1,0),(1,1)$, $m_X=\phi_X(b_1)=-1/6$ is already the global minimum of the function $\phi_X$ on $\B_K$. It follows that no semistable model of $X$ over the field $K$ exists (Proposition \ref{prop:stability_function_intro} (iii)).

So in order to find a semistable model, we have to pass to a suitable finite extension $L/K$. In a first attempt, we assume that  $L$ contains an element $\alpha$ such that $v_L(\alpha)=1/6$ and that $\alpha^6/2$ reduces to $1$. Then the point $b_1$ becomes a vertex of the building $\B_L$ and corresponds to a hypersurface model $\X_1$ of $X_L$. We then have to check whether the model $\X_1$ is semistable. If it is, we are done, otherwise we have to find a new apartment containing $b_1$ on which $\phi_X$ has a strictly smaller minimum than $-1/6$.

To write down an equation for the model $\X_1$, we note that $b_1=[v_{\E_1,0}]$, where
\[
\E_1:=\lexp{{g_1}}{{\E_0}} = (\alpha^{-3}x_0,\alpha^{-2}x_1,x_2), \quad \text{with}\;\;
g_1:= \begin{pmatrix} \alpha^3 & & \\ & \alpha^2 & \\ & & 1 \end{pmatrix}.
\]
We use the convention that the elements of the new basis $\E_1$ are again called $x_0,x_1,x_2$. Then $\X_1$ is given by the equation
\[
F_1 := \frac{1}{2}\cdot\lexp{{g_1}}{F}= \frac{1}{2}F(\alpha^3x_0,\alpha^2x_1,x_2) = \frac{\alpha^6}{2}x_0^2x_2-\frac{\alpha^6}{2}x_1^3+x_2^3 = 0.
\]
The special fiber $\X_{1,s}$ is given by the equation
\[
x_0^2x_2+x_1^3+x_2^3 = (x_0+x_2)^2x_2 + x_1^3 = 0.
\]
It is again a singular cubic with a double point at $P=[1:0:1]$ and  with tangent cone the double line $(x_0+x_2)^2=0$.  In particular, the model $\X_1$ is not semistable, and hence $\phi_X(b_1)=-1/6$ is not the minimal value of $\phi_X$ on $\B_L$. It follows that there exists a basis
\[
\E_2:=\lexp{{g_2}}{{\E_1}}, \quad \text{with $g_2\in\SL_3(\OO_L)$,}
\]
such that the minimal value $m_{\E_2}$ of $\phi_X$ on the apartment $\A_{\E_2}$ is $<-1/6$. In fact, the position of the singular point of $\X_{1,s}$ and its tangent cone shows that we can take
\[
g_2 := \begin{pmatrix} 1 & 0 & 0 \\ 0 & 1 & 0 \\ \beta & 0 & 1 \end{pmatrix},
\]
where $\beta\in L$ is any element with $v_L(\beta)=0$, $\bar{\beta}=1$. Indeed, the equation for $\X_1$ in terms of the new basis $\E_2$ is
\begin{equation} \label{eq:exa1.5}
	\begin{split}
		F_2 &:=\frac{1}{2}\cdot\lexp{{g_2g_1}}{F} = \lexp{{g_2}}{{F_1}} = F_1(x_0+\beta x_2,x_1,x_2) \\
		& = \frac{\alpha^6}{2}(x_0+\beta x_2)^2x_2-\frac{\alpha^6}{2}x_1^3+x_2^3 \\
		& = \frac{\alpha^6}{2}x_0^2x_2 +\alpha^6\beta x_0x_2^2 - \frac{\alpha^6}{2}x_1^3 + (1+\frac{\alpha^6\beta^2}{2})x_2^3.
\end{split}\end{equation}
Note that the coefficients before $x_0x_2^2$ and  $x_2^3$ have positive valuation, and the other coefficients have valuation $0$ and reduce to $1$.  Therefore, the new equation for the special fiber $\X_1$ is
\[
x_0^2x_2+x_1^3=0.
\]
Now the singular point is $P=[0:0:1]$ and its tangent cone $x_0^2=0$, as required by Proposition \ref{prop:plane_curves}.

Let $t:=v_L(\alpha^6\beta^2/2+1)>0$. Then by Lemma \ref{lem:stability_function1} (i) and \eqref{eq:exa1.5} we have
\[\begin{split}
	\phi_{F,\E_2}(w_0,w_1) &=  \phi_{F_2,\E_2}(w_0,w_1) - 1  \\
	&= \max\big( -w_0+w_1, w_1-1, w_0-2w_1, w_0+w_1-t\big) -\frac{1}{6}.
\end{split}\]
If $t<2$ then this function has a unique minimum at $(t/2,t/3)$, with value $-(t+1)/6>-1/2$. If $t\geq 2$ then it has a unique minimum at $b_3=(1,2/3)$, with value $-1/2$. Since we want to `minimize the minimum', we assume that $t\geq 2$. Moreover, we choose as our new basis
\[
\E_3 := \lexp{{g_3}}{{\E_2}}, \quad \text{with}\;\;
\begin{pmatrix} 2 & & \\  & \alpha^4 & \\ & & 1 \end{pmatrix},
\]
so that $b_3:=[v_{\E_3,0}]\in\B_L^\circ$. The model $\X_3$ of $X_L$ corresponding to $b_3$ has equation
\[\begin{split}
	F_3 &:= \frac{1}{4}\cdot F_2(2x_0,\alpha^4 x_1, x_2) \\
	& = \frac{\alpha^6}{2}x_0^2x_2 +\frac{\alpha^6\beta}{2}x_0x_2^2 - \frac{\alpha^{18}}{8}x_1^3 + cx_2^3,
\end{split}\]
where
\[
c := \frac{2+\alpha^6\beta^2}{8}.
\]
Note that $v_L(c)\geq 0$ because $t\geq 2$.
So the special fiber of $\X_3$  has equation
\[
x_0^2x_2 + x_0x_2^2 + x_1^3 + \bar{c}x_2^3= 0,
\]
which is the equation of a smooth cubic over the residue field of $L$. So $\X_3$ is a smooth and in particular stable model of $X_L$.

All in all, the equation for $\X_3$ was obtained by the transformation
\[
F_3 = \frac{1}{8}\cdot\lexp{g}{F}, \quad \text{with}\;\;
g = g_3g_2g_1 = \begin{pmatrix} 2\pi & 0 & 0 \\
	0 & \pi^2 & 0 \\ \pi & 0 & 1 \end{pmatrix},
\]
where $\pi:=\alpha^3$. We may set $\beta:=1$ provided that $v_L(2+\pi^2)\geq 3$. In hindsight, we see that we only need a quadratic extension $L/K$. For instance, we can choose  $L:=K(\pi\mid \pi^2+2=0)$. On the other hand, our computation shows that the existence of a semistable model implies the existence of an element $\pi\in\OO_L$ with $v_L(2+\pi^2)\geq 3$. By Hensel's Lemma this condition implies the existence of an element $\pi$ with $\pi^2+2=0$. Therefore, $L:=K(\sqrt{-2})$ is the unique quadratic extension which works.

Here is how this example is handled by the first author's code (\cite{KletusGitHub}):
\vspace{2ex}
\begin{Verbatim}[breaklines=true, breakanywhere=true]
	sage: from semistable_model.curves import PlaneCurveOverValuedField
	sage: # Define the polynomial ring and valuation
	sage: R.<x,y,z> = QQ[]
	sage: v_2 = QQ.valuation(2)
	sage: # Define the defining polynomial
	sage: F = x^2*z - y^3 + 2*z^3
	sage: # Initialize the curve
	sage: Y = PlaneCurveOverValuedField(F, v_2)
	sage: Y
	Projective Plane Curve with defining polynomial -y^3 + x^2*z + 2*z^3 over Rational Field with 2-adic valuation
	sage: X = Y.git_semistable_model(); X
	Plane Model of Projective Plane Curve with defining polynomial -y^3 + x^2*z + 2*z^3 over Number Field in piK with defining polynomial x^2 + 2 with 2-adic valuation
	sage: Xs = X.special_fiber(); Xs
	Projective Plane Curve with defining polynomial y^3 + x^2*z + x*z^2 over Finite Field of size 2
	sage: Xs.is_smooth()
	True
	sage: X.base_change_matrix()
	[-2*piK      0      0]
	[     0     -2      0]
	[   piK      0      1]
\end{Verbatim}
\vspace{2ex}

\subsection{A Ciani quartic over $\QQ_2^\nr$} \label{subsec:example3}

We consider the plane quartic $X : F = 0$ over $K=\QQ_2^\nr$ given by
\[
F = x_{1}^{4} + 2 x_{0}^{3} x_{2} + x_{0} x_{1}^{2} x_{2} + 2 x_{0} x_{2}^{3}.
\]
It is smooth and has two commuting involutions $\sigma,\tau$ as automorphisms, given by
\[
\sigma^*: (x_0,x_1,x_2) \mapsto (x_2,x_1,x_0), \quad
\tau^*: (x_0,x_1,x_2) \mapsto (x_0,-x_1,x_2).
\]
Curves with this property are called {\em Ciani curves}. The reduction behavior of Ciani curves over a local field has recently been studied in \cite{Ciani1}, \cite{Ciani2}, under the assumption that the residue characteristic is odd.

We will show that  $X$ has a stable plane model over a ramified quadratic extension $L/K$. Unlike for the cubic from the previous section \S \ref{subsec:example1}, the extension $L/K$ is not unique, but it is also not arbitrary. This reflects the fact that the stable plane model of $X$ is not geometrically semistable (unlike the cubic from \S \ref{subsec:example1}).

\vspace{2ex}
We use a similar approach and notation as in \S \ref{subsec:example1}, but give less details. In particular, we will not write down the formulae for the stability function.

One first checks that the minimum of the stability function on the apartment $\A_{\E_0}$ corresponding to the standard basis $\E_0=(x_0,x_1,x_2)$ is equal to $m_{\E_0}=-\frac{1}{3}$ and is achieved precisely at the point
\[
u_1 = \big( 0, \frac{1}{2}, 0 \big) \in\A_{\E_0}\subset \B_K.
\]
One can show that $m_{\E_0}$ is already the global minimum of the stability function on $\B_K$. As $u_1$ is not a vertex, it follows that $X$ does not have a plane semistable model over $K$. We will show this claim via a direct argument below.

The point $u_1$ becomes a vertex over any extension $L/K$ which contains an element $\alpha\in L$ with $v_L(\alpha)=1/2$. Over such an extension, the vertex $u_1$ corresponds to the $\OO_L$-model $\X_1$ of $X$ given by the equation
\[
F_1 := \frac{1}{2} F(x_0,\alpha x_1, x_2) = \frac{\alpha^4}{2} x_1^4 + x_0^3 x_2 + \frac{\alpha^2}{2} x_0 x_1^2 x_2 + x_0 x_2^3 = 0.
\]
For simplicity, we make the harmless assumption that $t:=v_L(\alpha^2+2)>1$ (this is not yet a restriction on the extension $L/K$). Then the special fiber $\X_{1,s}$ is given by the equation
\[
x_0^3x_2 + x_0x_1^2x_2 + x_0x_2^3 = x_0x_2(x_0+x_1+x_2)^2 = 0.
\]
It follows from Proposition \ref{prop:plane_curves} that $\X_{1,s}$ is unstable. In fact, Case (a) holds for the line $L:\;x_0+x_1+x_2=0$, and Case (c) and (d) both hold for the pair $(L,P)$, where $L$ is as before, and $P$ is either one of the two points $[0:1:1]$ and $[1:1:0]$.

The instability of $\X_{1,s}$ corresponding to the line $L$ suggest the coordinate transformation
\[
F_2 := F_1(x_0,x_0+x_1+x_2,x_2).
\]
This is the equation for the model $\X_1$ with respect to the basis
\[
\E_2 := \lexp{g}{{\E_0}}, \quad g := g_2g_1 =
\begin{pmatrix}
	1 & \alpha & 0 \\ 0 & \alpha & 0 \\ 0 & \alpha & 1
\end{pmatrix}.
\]
It is chosen so that the line $L$ responsible for the instability of $\X_{1,s}$ is given by the equation $x_1=0$.

The point $u_1$ lies at the boundary of the intersection $\A_{\E_0}\cap\A_{\E_2}$. By the choice of $\E_2$, $u_1$ is not the minimum of the stability function on the apartment $\A_{\E_2}$, i.e. we have $m_{\E_2}<m_{\E_1}=m_{\E_0}$. 

Let us now assume that $t=v_L(\alpha^2+2)\geq 2$; e.g.\ we could choose $\alpha:=\sqrt{-2}$. Then one computes that the new minimum is  $m_{\E_2}=-2/3$, which is achieved precisely in the point $u_3=(0,1/2,0)\in\A_{\E_2}$. The model $\X_3$ corresponding to $u_3$ has equation
\[
F_3 := \frac{1}{2}F_2(x_0,\alpha x_1,x_2).
\]
This works because the coefficients of $x_0^3x_2$ and $x_0x_2^3$ in $F_3$ are equal to
\[
\alpha^4+\frac{1}{4}\alpha^2+\frac{1}{2} = \alpha^4 +\frac{1}{4}(\alpha^2+2), 
\]
a term which has valuation $\geq \min(2,t-2) \geq 0$. If we assume, moreover, that $t>2$ then these two coefficients vanish after reduction and the special fiber $\X_{3,s}$ of the model $\X_3$ has equation
\[
\bar{F}_3 = x_0^4 +  x_0x_1^2x_2 + x_0^2x_2^2 + x_2^4 = 0.
\]
By an easy computation one shows that $\X_{3,s}$ is reduced and irreducible and has exactly three singular points, each of multiplicity $2$:
\[
P_1=[0:1:0], \quad P_2=[\bar{\zeta}:0: 1], \quad P_3=[\bar{\zeta}+1:0:1].
\]
Here $\bar{\zeta}\in k$ is a generator of the subextension $\FF_4/\FF_2$.

The point $P_1$ is an ordinary double point (its tangent cone is given by $x_0x_1=0$), but the points $P_2,P_3$ are cusps, with tangent cone $(x_0+\bar{\zeta}x_1)^2=0$ resp.\ $(x_0+(\bar{\zeta}+1)x_1)^2=0$. It follows from Proposition \ref{prop:stability_function_intro} that $\X_{3,s}$ is semistable. In fact, a closer look at the proof of Proposition\ref{prop:plane_curves}, or at \cite[\S 4.2]{MumfordGIT} shows that $\X_{3,s}$ is stable (this is also a special case of the main result of \cite{Mordant2023}). It follows that $\X_3$ is the unique stable plane model of $X$ over the extension $L/K$. In particular, the point $u_3=(0,1/2,0)\in\A_{\E_2}\subset\B_L$ is the unique point where the stability function achieves its global minimum $-3/2$.

The only property of $L$ that we have used for the construction of the model $\X_3$ is that it contains an element $\alpha$ with $v_L(\alpha^2+2)\geq 2$. Among the three ramified quadratic extensions of $K=\QQ_2^\nr$, two have this property (namely $K_1:=K(\sqrt{2})$ and $K_2:=K(\sqrt{-2})$), but the third, $K_3:=K(\sqrt{-1})$, has not. One computes that the global minimum of the stability function on $\B_{K_3}$ is equal to $-1/2$, and is not achieved in a vertex. We conclude that the curve $X$ has no plane semistable model over the extension $K_3/K$.


Note that the model $\X_3$ is not {\em geometrically} semistable because the special fiber $\X_{3,s}$ has two cusps as singularities. Assume that the extension $L/K$ has been chosen sufficiently large so that $X$ has geometric semistable reduction over $L$. Let $\X$ be the stable model (in the sense of Deligne and Mumford) of $X$ over $\OO_L$. Then the main result of \cite{MaxMaster} predicts the following:
\begin{enumerate}[(i)]
	\item
	The geometric stable model $\X$ dominates the plane stable model $\X_3$, i.e.\ there exists a (unique) morphism $\X\to\X_3$ of $\OO_L$-schemes extending the identity on the generic fiber $X$.
	\item
	The special fiber $\X_s$ has three irreducible components $Z_1,Z_2,Z_3$. The curve $Z_1$ is semistable of geometric genus zero and has a single ordinary double point. The curves $Z_2,Z_3$ are semistable curves of genus one, and each intersects $Z_1$ in a unique point (i.e.\ $Z_2,Z_3$ are so-called \emph{$1$-tails}).
	\item
	The induced map $\X_s\to\X_{3,s}$ restricts to a homeomorphism $Z_1\iso \X_{3,s}$, and contracts the components $Z_2$, $Z_3$ to the cusps $P_2,P_3\in\X_{3,s}$.
\end{enumerate}

In our follow-up article \cite[\S 5.1]{SSW} we show how to compute the extension $L/K$ and the modification $\X\to\X_3$ giving rise to the stable model $\X$ explicitly. The outcome in this particular example is the stable model $\X$ can be defined over a Galois extension $L/K$ of degree $24$. The two $1$-tails $Z_2,Z_3$ are both isomorphic to the smooth cubic in Weierstrass normal form 
\[
y^2 + y = x^3.
\]  
The Galois group of $L/K$ acts naturally and faithfully on $\X_s$. This action is trivial on the strict transform of $\X_{3,s}$,  fixes the two tails $Z_2,Z_3$ and acts on each of them  as the full automorphism group (which is isomorphic to $\SL_2(3)$).

A slightly inconvenient observation here is that the extension $L/K$ described above does not contain any of the two quadratic extensions $K_1=\QQ_2(\sqrt{2})$ or $K_2=\QQ_2(\sqrt{-2})$ of $K$. So if we first compute the GIT-stable model $\X_3$ over one of these quadratic extensions, say $K_1$, then we need a further extension $L_1/K_1$ of degree $24$ to obtain geometric semistable reduction. The extension $L_1/K$ then has degree $48$, which is not minimal. This shows that our two-step strategy (first a GIT-semistable plane model, then the
geometric stable model) need not produce a minimal extension.

\subsection{A quartic with infinitely many semistable models} \label{subsec:example2}

Let $K$ be a field which is complete with respect to a discrete valuation $v_K$, and assume that the residue field $k$ of $v_K$ is infinite. We will give an example of a smooth plane quartic over $K$ which has infinitely many nonisomorphic semistable models. This realizes concretely the phenomenon mentioned at the end of \S \ref{subsec:stability_function_intro}.

Let $\pi$ denote a prime element of $K$, and let $G\in \OO_K[x_0,x_1,x_2]_4$ be a form of degree $4$ with integral coefficients. We set
\[
X:=V_+(F), \quad F:=(x_0^2+x_1x_2)^2 + \pi^5G.
\]
One can show that we can choose $G$ such that $X$ is smooth. For instance, if ${\rm char}(k)\neq 2$, then it suffices that the conic $Q$ meets the quartic $G=0$ in $8$ distinct geometric points (see e.g.\ \cite[Proposition 1.2]{LLLR}). The model $\X_0$ of $X$ corresponding to the standard coordinate system $\E_0=(x_0,x_1,x_2)$ has special fiber $\X_{0,s}=V_+(\overline{F})$, with
\[
\overline{F}=(x_0^2+x_1x_2)^2 = x_0^4 + 2x_0^2x_1x_2 + x_1^2x_2^2,
\]
see Example \ref{exa:semi-instabilities}. This is a semistable, but not properly stable quartic over $k$. In particular, $\X_0$ is a semistable model of $X$.

Let $t\in \OO_K$ be an integral element of $K$, and set
\[
g_t := \begin{pmatrix} 1 & & \\ & \pi^{-1} & \\ & & \pi \end{pmatrix} \cdot \begin{pmatrix} 1 & 0 & -2t \\ t & 1 & -t^2\\ 0 & 0 & 1 \end{pmatrix} =
\begin{pmatrix}
	1 & 0 & -2t \\ t\pi^{-1} & \pi & -t^2\pi^{-1} \\ 0 & 0 & \pi
\end{pmatrix}.
\]
An easy calculation shows that
\[
\lexp{{g_t}}{F} = (x_0^2+x_1x_2)^2 + \pi\tilde{G},
\]
where $\tilde{G}:=\pi^4\cdot\lexp{{g_t}}{G}$ is a quartic form with integral coefficients. The choice of $g_t$ is informed by the reasoning in Example \ref{exa:semi-instabilities}: we  transform the standard coordinate system $(x_0,x_1,x_2)$ to a new coordinate system $(y_0,\pi^{-1} y_1,\pi y_2)$, where the line $y_0=0$ passes through the two points
\[
P_1=[0:0:1], \quad P_2=[t:1:-t^2]
\]
on the conic $C:x_0^2-x_1x_2$, and $y_1=0$ and $y_2=0$ are the tangent lines to $C$ through $P_1$ and $P_2$, respectively. This coordinate change preserves the equation $x_0^2-x_1x_2$ of $C$.

In any case, we see that the model $\X_t:=\lexp{{g_t}}{{\X_0}}$ is also semistable. We claim that two such models $\X_t$ and $\X_s$, with $s,t\in\OO_K$, are non-isomorphic if the images of $t$ and $s$ in the residue field $k$ are distinct. To see this, it suffices to check that the matrix
\[
g_tg_s^{-1} = \begin{pmatrix}
	1 & 0 & 2(s-t)\pi^{-1} \\ (t-s)\pi^{-1} & 1 & -(t-s)^2\pi^{-2} \\
	0 & 0 & 1 \end{pmatrix}
\]
does not lie in $\GL_3(\OO_K)$. This is clearly true if $s\not\equiv t\pmod{\pi}$. We conclude that there are infinitely many non-isomorphic semistable models of $X$ if the residue field $k$ is infinite.

\vspace{2ex}
It is interesting to consider the above example in the light of the articles \cite{LLLR}, \cite{MaxMaster} and \cite{SSW}. The fact that the GIT-semistable models $\X_t$ constructed above are not properly stable implies, by the main result of \cite{MaxMaster}, that the curve $X$ has {\em hyperelliptic reduction}. This means that the special fiber $\X_s$ of the geometric stable model $\X$ of $X$ lies in the closure of the hyperelliptic locus of $\overline{\M}_3$. And in general, the stable model may not dominate any of the GIT-semistable plane models $\X_t$ constructed above. 

However, if we choose the form $G$ generically, then \cite[Theorem 1.4]{LLLR} shows that the special fiber $\X_s$ is a smooth hyperelliptic curve. Moreover, $\X$ dominates a unique semistable plane model $\X_0$, whose special fiber is a double conic. So the geometric stable model of $X$ `picks out` one of the infinitely many plane semistable models. For our motivating question, this is bad news. It shows that in the case of hyperelliptic reduction, our strategy of `compute a GIT-semistable model first and then the geometric semistable model`, worked out in \cite{SSW}, doesn't work as well as in the case of non-hyperelliptic reduction. We hope to come back to this problem  in a future article.

\vspace{2ex}\noindent
{\bf Data availability statement:} No new data were generated or analyzed in support of this research. Therefore, data sharing is not applicable.

\vspace{2ex}\noindent
{\bf Conflict of interest statement:}
The authors declare that they have no conflict of interest.


\end{document}